\documentclass[a4paper,11pt]{amsart}
\usepackage[utf8]{inputenc}
\usepackage[T1]{fontenc}
\usepackage[english]{babel}
\usepackage{enumitem}
\usepackage{esint}
\usepackage{lmodern}
\usepackage{geometry} 
\usepackage{amsbsy} 
\usepackage{mathrsfs}
\usepackage{color}

\newcommand \Ln{\mathscr{L}^n}
\newcommand \Hs{\mathscr{H}^{n-1}}
\newcommand \Om{\Omega}
\newcommand \om{\omega}
\newcommand \eps{\epsilon}
\newcommand \Uk{\mathcal{U}_k}
\newcommand \Up{\mathcal{U}_p}
\newcommand \Per{\text{Per}}
\newcommand \R{\mathbb{R}}
\newcommand \Rn{\mathbb{R}^n}
\newcommand \Span{\text{span}}
\newcommand \Tr{\text{Tr}}
\newcommand \F{\mathcal{F}}
\newcommand \B{\mathbb{B}}
\newcommand \U{\textbf{u}}
\newcommand \V{\textbf{v}}
\newcommand \W{\textbf{w}}
\newcommand \SBV{\mathrm{SBV}}

\newtheorem{main} {Theorem} 
\newtheorem{proposition}[main] {Proposition}
\newtheorem{lemma}[main]{Lemma}  
\theoremstyle{definition} 
\newtheorem{definition}[main] {Definition}

\numberwithin{equation}{section}

\begin{document}

\title[M. Nahon]{Existence and regularity of optimal shapes for spectral functionals with Robin boundary conditions}
\author[M. Nahon]{Mickaël Nahon}
\address[Mickaël Nahon]{Univ. Savoie Mont Blanc, CNRS, LAMA \\ 73000 Chamb\'ery, France}
\email{  mickael.nahon@univ-smb.fr}

\maketitle

\begin{abstract}
We establish the existence and find some qualitative properties of open sets that minimize functionals of the form $ F(\lambda_1(\Om;\beta),\hdots,\lambda_k(\Om;\beta))$ under measure constraint on $\Om$, where $\lambda_i(\Om;\beta)$ designates the $i$-th eigenvalue of the Laplace operator on $\Om$ with Robin boundary conditions of parameter $\beta>0$. Moreover, we show that minimizers of $\lambda_k(\Om;\beta)$ for $k\geq 2$ verify the conjecture $\lambda_k(\Om;\beta)=\lambda_{k-1}(\Om;\beta)$ in dimension three and more.
\end{abstract}

\tableofcontents

\section{Introduction}

Let $\Om$ be a bounded Lipschitz domain in $\Rn$, $\beta>0$ a parameter that is constant throughout the paper, and $f\in L^2(\Om)$. The Poisson equation with Robin boundary conditions is
\[\begin{cases}-\Delta u=f& \text{ in }\Om,\\ \partial_\nu u+\beta u=0 & \text{ in }\partial\Om,\end{cases}\]
where $\partial_\nu$ is the outward normal derivative that may only have a meaning in the sense that for all $v\in H^1(\Om)$,
\[\int_{\Om}\nabla u\cdot\nabla v\mathrm{d}\Ln+\int_{\partial\Om}\beta u v\mathrm{d}\Hs=\int_{\Om}fv\mathrm{d}\Ln.\]
This equation (and in particular its boundary conditions) has several interpretations: we may see the solution $u$ as the temperature obtained in an homogeneous solid $\Om$ with the volumetric heat source $f$, and insulator on the boundary (more precisely, a width $\beta^{-1}\epsilon$ of insulator of conductivity $\epsilon$ for $\epsilon\rightarrow 0$) that separates the solid $\Om$ from a thermostat.\\
Another interpretation is to see $u$ as the vertical displacement of a membrane with shape $\Om$ on which we apply a volumetric normal force $f$, and the membrane is fixed on its boundary by elastic with stiffness proportional to $\beta$.\\

This equation is associated to a sequence of eigenvalues
\[0<\lambda_1(\Om;\beta)\leq \lambda_2(\Om;\beta)\leq \hdots\rightarrow +\infty,\]
with eigenfunctions $u_k(\Om;\beta)$ that verify
\[\begin{cases}\Delta u_k(\Om;\beta)+\lambda_{k}(\Om;\beta)u_k(\Om;\beta)=0& \text{ in }\Om,\\ \partial_\nu u_k(\Om;\beta)+\beta u_k(\Om;\beta)=0 & \text{ in }\partial\Om.\end{cases}\]

The quantities $(\lambda_k(\Om;\beta))_k$ may be extended to any open set $\Om$ in a natural way, see Section 2 for more details.\bigbreak

In this paper, we study some shape optimization problems involving the eigenvalues $(\lambda_k(\Om;\beta))_k$ with measure constraint on general open sets. In particular we prove that when $F(\lambda_1,\hdots,\lambda_k)$ is a function with positive partial derivative in each $\lambda_i$ (such as $F(\lambda_1,\hdots,\lambda_k)=\lambda_1+\hdots+\lambda_k$), then for any $m,\beta>0$ the optimisation problem
\[\min\left\{F\left(\lambda_1(\Om;\beta),\hdots,\lambda_k(\Om;\beta)\right),\ \Omega\subset\Rn\text{ open such that }|\Om|=m\right\}\]
has a solution. Moreover the topological boundary of an optimal set is rectifiable, Ahlfors-regular, with finite $\Hs$-measure. 
For functionals of the form $F(\lambda_1,\hdots,\lambda_k)=\lambda_k$, while minimizers are only known to exist in a relaxed $\SBV$ setting (that will be detailed in the second section), we show that any $\SBV$ minimizer verifies
\[\lambda_k(\Om;\beta)=\lambda_{k-1}(\Om;\beta)\]
in any dimension $n\geq 3$.
\subsection{State of the art}

The link between the eigenvalues of the Laplace operator (or other differential operators) on a domain and the geometry of this domain is a problem that has been widely studied, in particular in the field of spectral geometry.\\

The earliest and most well-known result in this direction dates back to the Faber-Krahn inequality, that states that the first eigenvalue of the Laplacian with Dirichlet boundary conditions is, among sets of given measure, minimal on the disk. The same result was shown for Robin boundary conditions with positive parameter in \cite{B88} in the two-dimensional case, then in \cite{D06} in any dimension for a certain class of domains on which the trace may be defined, using dearrangement methods. It was extended in \cite{BG10}, \cite{BG15} in the $\SBV$ framework that we will describe in the next section, such that the first eigenvalue with Robin boundary condition is minimal on the ball among all open sets of given measure. In order to handle the lack of uniform smoothness of the admissible domains, the method here is to consider a relaxed version of the problem, so as to optimize an eigenfunction instead of a shape. Once it is known a minimizer exists in the relaxed framework, it is shown by regularity and symmetry arguments that this minimizer corresponds to the disk.\\

Similar problems of spectral optimization with Neumann boundary conditions or Robin conditions with negative parameter have been shown to be different in nature, in the former case the first eigenvalue is maximal on the disk, and this is shown with radically different method, mainly building appropriate test functions since the eigenvalues are defined as an infimum through the Courant-Fischer min-max formula. Let us also mention several maximization result for Robin boundary condition with parameter that scales with the perimeter, obtained in \cite{L19}, \cite{GL19} with similar methods.\\

The existence and partial regularity for minimizers of functions $F(\lambda_1^D(\Om),\hdots,\lambda_k^D(\Om))$ (where $\lambda_i^D(\Om)$ is the $i$-th eigenvalue of the Laplacien with Dirichlet boundary conditions) with measure constraint or penalization has been achieved in \cite{B12}, \cite{MP13}, \cite{KL18}, \cite{KL19}: it is known that if $F$ is increasing and bi-Lipschitz in each $\lambda_i$ then there is an optimal open set that is $\mathcal{C}^{1,\alpha}$ outside of a singular set of codimension at least three, and if $F$ is merely nondecreasing in each coordinate then there is an optimal quasiopen set that has analytic boundary outside of a singular set of codimension three and points with Lebesgue density one. It has been shown in \cite{BMPV15}, \cite{KL19} that a shape optimizer for the $k$-th eigenvalue with Dirichlet boundary conditions and measure constraint admits Lipschitz eigenfunctions. In these papers the monotonicity and scaling properties of the eigenvalues with Dirichlet boundary condition ($\om\mapsto \lambda_k^D(\om)$ is decreasing in $\om$) plays a crucial role, however eigenvalues with Robin boundary conditions have no such properties so the same methods cannot be extended in a straightforward way.\\

The minimization of $\lambda_2(\Om;\beta)$ under measure constraint on $\Om$ was treated in \cite{K09}; as in the Dirichlet case, the minimizer is the disjoint union of two balls of same measure. For the minimization of $\lambda_k(\Om;\beta)$ or other functionals of $\lambda_1(\Om;\beta),\hdots,\lambda_k(\Om;\beta)$, nothing is known except for the existence of a minimizer in the relaxed setting with bounded support for $\lambda_k(\Om;\beta)$, see \cite{BG19}. The regularity theory for minimizers of functionals involving Robin boundary conditions was developped in \cite{CK16}, \cite{K19} and we will relay on some of its results in our vectorial setting.\\

Numerical simulations in \cite{AFK13} for two-dimensional minimizers of $\lambda_k(\cdot;\beta)$ (for $3\leq k\leq 7$) with prescribed area suggest a bifurcation phenomena in which the optimal shape is a union of $k$ balls for every small enough $\beta$, and it is connected for any large enough $\beta$. In \cite{AFK13}, the connected minimizers were searched by parametric optimization among perturbations of the disk, however a consequence of our analysis in the last section is that minimizers of $\lambda_3(\cdot;\beta)$ are never homeomorphic to the disk.

\subsection{Statements of the main results}

In the first part of the paper, we are concerned in what we call the non-degenerate case; consider
\[F:\left\{\lambda\in\R^k:0<\lambda_1\leq \lambda_2\leq\hdots\leq \lambda_k\right\}\to\R_+\]
a Lipschitz function with directional derivatives - in the sense that for any $\lambda\in\R^k$ there is some positively homogeneous function $F_0$ such that $F(\lambda+\nu)=F(\lambda)+F_0(\nu)+o_{\nu\to 0}(|\nu|)$ - such that for all $i\in\left\{1,\hdots,n\right\}$,  and all $0<\lambda_1\leq \hdots\leq\lambda_k$
\begin{equation}\label{HypF}
\frac{\partial F}{\partial^\pm \lambda_i}(\lambda_1,\hdots,\lambda_k)>0,\ F(\lambda_1,\hdots,\lambda_{k-1},\mu_k)\underset{\mu_k\to\infty}{\longrightarrow}+\infty,
\end{equation}
where $\frac{\partial}{\partial^{\pm} \lambda_i}$ designates the directional partial derivatives in $\lambda_i$. This applies in particular to any of these:
\[F_p(\lambda_1,\hdots,\lambda_k)=\left(\sum_{i=1}^k \lambda_i^p\right)^\frac{1}{p}.\]
Our first main result is the following.

\begin{main}\label{main1}
Let $F$ be such a function, $m>0$, then there exists an open set that minimizes the functional
\[\Om\mapsto F(\lambda_1(\Om;\beta),\hdots,\lambda_k(\Om;\beta))\]
among open sets of measure $m$ in $\Rn$. Moreover any minimizing set is bounded, verifies $\Hs(\partial\Om)\leq C$ for some constant $C>0$ depending only on $(n,m,\beta,F)$, and $\partial\Om$ is Ahlfors-regular.
\end{main}

Here are the main steps of the proof:
\begin{itemize}[label=\textbullet]
\item \textbf{Relaxation}. We relax the problem in the $\SBV$ framework; this is introduced in the next subsection, following \cite{BG10}, \cite{BG15}, \cite{BG19}. The idea is that the eigenfunctions on a domain $\Om$ are expected to be zero almost nowhere on $\Om$; we extend these eigenfunctions by zero outside of $\Omega$ (thereby creating a discontinuity along $\partial\Om$) and reformulate the optimization problem on general functions defined in $\R^n$ that may have discontinuities, with measure constraint on their support. The advantage is that we have some compactness and lower semi-continuity results to obtain the existence of minimizers in the relaxed framework, however a sequence of eigenfunctions extended by zero may converge to a function that does not correspond to the eigenfunction of an open domain, so we will need to show some regularity on relaxed minimizers.
\item \textbf{A priori estimates and nondegeneracy}. We obtain a priori estimates for relaxed interior minimizers (meaning minimizers compared to any set that it contains). More precisely for any interior minimizer that corresponds to the eigenfunctions $(u_1,\hdots,u_k)$, we show that for almost any point in the support of these eigenfunctions, at least one of them is above a certain positive threshold. We also obtain $L^\infty$ bounds of these eigenfunctions and deduce a lower estimate for the Lebesgue density of the support, from which we obtain the boundedness of the support.

\item \textbf{Existence of minimizers}. We consider a minimizing sequence and show that, up to a translation, it either converges to a minimizer or it splits into two minimizing sequences of similar functionals depending on $p$ and $k-p$ (where $1\leq p<k$) eigenvalues respectively, and we know minimizers of these exists by induction on $k$.
\item \textbf{Regularity}. Finally, we show the regularity of relaxed minimizer, meaning that a relaxed minimizer corresponds to the eigenfunctions of a certain open domain that were extended by zero, by showing that the singular set of relaxed minimizers is closed up to a $\Hs$-negligible set.
\end{itemize}

Notice that in the second step we do not show that the first eigenfunction (or one of the $l$ first in the case where the minimizer has $l$ connected components) is positive, which is what we expect in general for sufficiently smooth sets; if $u_1$ is the first (positive) eigenfunction on a connected $\mathcal{C}^2$ set $\Om$, and suppose $u_1(x)=0$ for some $x\in\partial\Om$ then by Hopf's lemma $\partial_\nu u_1(x)<0$, which breaks the Robin condition at $x$, so $\inf_{\Om}u_1>0$. In our case, we get instead a "joint non-degeneracy" of the eigenfunctions in the sense that at every point of their joint support, at least one is positive.\bigbreak

Notice also that the second hypothesis in \eqref{HypF} is not superfluous: without it, a minimizing sequence $(\Om^i)$ could have some of its first $k$ eigenvalues diverge. This is because, unlike the Dirichlet case, there is no upper bound for $\frac{\lambda_k(\cdot;\beta)}{\lambda_1(\cdot;\beta)}$ in general even among sets with fixed measure. While $\lambda_k(\cdot;\beta)$ is not homogeneous by dilation, we still have the scaling property
\[\lambda_k(r\Om;\beta)=r^{-2}\lambda_k(\Om;r\beta).\]
Consider a connected smooth open set $\Om$. Since each $\lambda_k(\Om;r\beta)$ converges to $\lambda_k(\Om,0)$ (the eigenvalues with Neumann boundary conditions) as $r\to 0$, and $0=\lambda_1(\Om,0)<\lambda_2(\Om,0)$, then for any $k\geq 2$, $\frac{\lambda_k(r\Om;\beta)}{\lambda_1(r\Om;\beta)}\underset{r\to 0}{\longrightarrow}+\infty$. A counterexample among sets of fixed measure may be obtained with the disjoint union of $r\Om$ for small $r$ and a set $\om$ with prescribed measure such that $\lambda_1(\om;\beta)>\lambda_k(r\Om;\beta)$, such as a disjoint union of enough balls of radius $\rho>0$, chosen small enough to have $\lambda_1(\om,\beta)=\lambda_1(\B_\rho;\beta)>\lambda_k(r\Om;\beta)$.\bigbreak

In the second part of the paper, we study the minimizers of the functional
\[\Om\mapsto \lambda_k(\Om;\beta).\]
A minimizer in the $\SBV$ framework (see the introduction below) was shown to exist in \cite{BG19}, and aside from the fact that its support is bounded nothing more is known. We show that, in this $\SBV$ framework, a minimizer necessarily verify that $\lambda_{k-1}(\Om;\beta)=\lambda_k(\Om;\beta)$, in the context of definition \ref{def_relaxed}.\bigbreak
This is a long lasting open problem for minimizers of $\lambda_k$ with Dirichlet boundary condition (see \cite[open problem 1]{H06} and \cite{O04}).\bigbreak

However, although we prove it for Robin conditions, we do not expect this result to directly extend to the Dirichlet case ; simply put, even if some smooth sequence of minimizers $\Om^\beta$ of $\lambda_k(\cdot;\beta)$ approached a minimizer $\Om$ of $\lambda_k^D$ that is a counterexample of the conjecture, then there is no reason why the upper semi-continuity $\lambda_{k-1}^D(\Om)\geq \limsup_{\beta\to\infty}\lambda_{k-1}(\Om^\beta;\beta)$ should hold.

\begin{main}
Suppose $n\geq 3$, $k\geq 2$. Let $m>0$, and let $\U$ be a relaxed minimizer of
\[\V\mapsto \lambda_k(\V;\beta)\]
among admissible functions with support of measure $m$. Then
\[\lambda_{k-1}(\U;\beta)=\lambda_k(\U;\beta).\]
\end{main}

Here are the main steps and ideas of the proof:
\begin{itemize}[label=\textbullet]
\item First, we replace the minimizer $\U=(u_1,\hdots,u_k)$ with another minimizer $\V=(v_1,\hdots,v_k)$, with the property that $v_1\geq 0$, $\lambda_{k}(\V;\beta)>\lambda_{k-1}(\V;\beta)$, and $\V\in L^\infty(\R^n)$. One might think this estimate also holds for $\U$, however there is no particular reason why $\Span(\U)$ should contain eigenfunctions for $\lambda_1(\U;\beta),\hdots,\lambda_{k-1}(\U;\beta)$ in a variational sense.\\
This phenomenon may be easily understood in a finite-dimensional setting as follows: consider the matrix
$A=\begin{pmatrix}\lambda_1 \\ & \lambda_2 \\ & & \lambda_3\end{pmatrix}$ with $\lambda_1<\lambda_2<\lambda_3$. Then $\Lambda_2\Big[A\Big]$ is given by:
\[\lambda_2=\inf_{V\subset \R^3,\text{dim}(V)=2}\sup_{x\in V}\frac{(x,Ax)}{(x,x)}.\]
This infimum is reached by the subspace $\Span(e_1,e_2)$, but also by any subspace $\Span(e_1+te_3,e_2)$ for $|t|\leq \frac{\lambda_2-\lambda_1}{\lambda_3-\lambda_2}$, and these subspaces do not contain the first eigenvector $e_1$.
\item Then we obtain a weak optimality condition on $u_k$ using perturbations on sets with a small enough perimeter. The reason for this is that we have no access to any information on $u_1,\hdots,u_{k-1}$ apart from the fact that their Rayleigh quotient is strictly less than $\lambda_k(\Om;\beta)$, so we must do perturbations of $u_k$ that do not increase dramatically the Rayleigh quotient of $u_1,\hdots,u_{k-1}$.
\item We apply this to sets of the form $\B_{x,r}\cap\left\{|u_k|\leq t\right\}$ where $r$ is chosen small enough for each $t$. With this we obtain that $|u_k|\geq c 1_{\left\{ u_k\neq 0\right\}}$.
\item We deduce the result by showing that the support of $\U$ is disconnected, so the $k$-th eigenvalue may be decreased without changing the volume by dilations.
\end{itemize}

While the existence of open minimizers is not yet known, we end with a few observations on the topology of these minimizer, in particular with the fact that a bidimensionnal minimizer of $\lambda_3(\cdot;\beta)$ with prescribed measure is never simply connected.

\section{Relaxed framework}

Throughout the paper, we use the relaxed framework of $\SBV$ functions to define Robin eigenvalues on any open set without regularity condition, and more importantly to transform our shape optimization problem into a free discontinuity problem on functions that are not defined on a particular domain any more. The $\SBV$ space was originally developed to handle relaxations of free discontinuity problems such as the Mumford-Shah functional that will come into play later, we refer to \cite{AFP00} for a complete introduction. $\SBV$ functions may be thought of as "$W^{1,1}$ by part" functions, and this space is defined as a particular subspace of $BV$ as follows: 
\begin{definition}
A $\SBV$ function is a function $u\in BV(\Rn,\R)$ such that the distributional derivative $Du$ (which is a finite vector-valued Radon measure) may be decomposed into
\[Du=\nabla u \Ln +(\overline{u}-\underline{u})\nu_u \Hs_{\lfloor J_u}, \]
where $\nabla u\in L^1(\Rn)$, $J_u$ is the jump set of $u$ defined as the set of point $x\in\Rn$ for which there is some $\overline{u}(x)\neq \underline{u}(x)\in\R$, $\nu_u(x)\in\mathbb{S}^{n-1}$, such that
\[\left(y\mapsto u(x+ry)\right)\underset{L^1_{\text{loc}}(\Rn)}{\longrightarrow}\overline{u}(x)1_{\{y:y\cdot\nu_u(x)>0\}}+\underline{u}(x)1_{\{y:y\cdot\nu_u(x)<0\}}\text{ as }r\to 0.\]
\end{definition}

We will not work directly with the $\SBV$ space but with an $L^2$ analog defined below, that was studied in \cite{BG10}.

\begin{definition}
Let $\Uk$ be the space of functions $\U\in L^2(\R^n,\R^k)$ such that
\[D\U=\nabla \U \Ln +(\overline{\U}-\underline{\U})\nu_\U \Hs_{\lfloor J_\U}, \]
where $\nabla \U\in L^2(\R^n,\R^{nk})$ and $\int_{J_\U}(|\overline{\U}|^2+|\underline{\U}|^2)\mathrm{d}\Hs<\infty$. The second term will be written $D^s\U$ ($s$ stands for singular). The function $\U$ is said to be linearly independant if its components span a $k$-dimensional space of $L^2(\R^n)$.\\
We will also say that a function $\U\in \Uk$ is disconnected if there is a measurable partition $\Om,\om$ of the support of $\U$ such that $\U1_\Om$ and $\U1_\om$ are in $\Uk$, and: \begin{align*}
D^s(\U1_\Om)&=(\overline{\U1_\Om}-\underline{\U1_\Om})\nu_\U \Hs_{\lfloor J_\U},\\
D^s(\U1_\om)&=(\overline{\U1_\om}-\underline{\U1_\om})\nu_\U \Hs_{\lfloor J_\U}.\end{align*}
In this case we will write $\U=(\U1_\Om)\oplus(\U1_\om)$.
\end{definition}

The following compactness theorem is a reformulation of Theorem 2 from \cite{BG10}.
\begin{proposition}\label{CompactnessLemma}
Let $(\U^i)$ be a sequence of $\Uk$ such that
\[\sup_{i}\int_{\R^n}|\nabla\U^i|^2\mathrm{d}\Ln+\int_{J_\U}(|\underline{\U^i}|^2+|\overline{\U^i}|^2)\mathrm{d}\Hs+\int_{\Rn}|\U^i|^2\mathrm{d}\Ln<\infty,\]
then there exists a subsequence $(\U^{\phi(i)})$ and a function $\U\in\Uk$ such that
\begin{align*}
\U^{\phi(i)}&\underset{L^2_\text{loc}}{\longrightarrow}\U,\\
\nabla\U^{\phi(i)}&\underset{L^2_{\text{loc}}-\text{weak}}{\rightharpoonup}\nabla\U,\\
\end{align*}
Moreover for any bounded open set $A\subset\R^n$
\begin{align*}
\int_{A}|\nabla \U|^2\mathrm{d}\Ln&\leq \liminf_{i\to+\infty}\int_{A}|\nabla \U_\phi(i)|^2\mathrm{d}\Ln,\\
\int_{J_u\cap A}(|\overline{\U}|^2+|\underline{\U}|^2)\mathrm{d}\Hs&\leq \liminf_{i\to+\infty}\int_{J_u\cap A}(|\overline{\U^{\phi(i)}}|^2+|\underline{\U^{\phi(i)}}|^2)\mathrm{d}\Hs.\\
\end{align*}
\end{proposition}
\begin{proof}
The proof is an adaptation of \cite[theorem 2]{BG10} to a multidimensional case.
\end{proof}
We define a notion of $i$-th eigenvalue of the Laplace operator with Robin boundary conditions that allows us to speak of the functional $\lambda_k(\cdot;\beta)$ with no pre-defined domain, and to define the $k$-th eigenvalue on any open set even when the trace of $H^1$ functions is not well-defined.

\begin{definition}\label{def_relaxed}
Let $\U\in\Uk$ be linearly independant. We define the two Gram matrices:
\begin{align*}\label{defmatrix}
A(\U)&=\left(\langle u_i,u_j\rangle_{L^2(\mathbb{R}^n,\Ln)}\right)_{1\leq i,j\leq k},\\
B(\U)&=\left(\langle \nabla u_i,\nabla u_j\rangle_{L^2(\mathbb{R}^n,\Ln)}+\beta\langle \overline{u_i},\overline{u_j}\rangle_{L^2(J_\U,\Hs)}+\beta\langle \underline{u_i},\underline{u_j}\rangle_{L^2(J_\U,\Hs)}\right)_{1\leq i,j\leq k}.
\end{align*}
We then define the $i$-th eigenvalue of the vector-valued function $\U$ as

\begin{equation}
\lambda_i(\U;\beta)=\inf_{V\subset\Span(\U), \dim(V)=i}\sup_{v\in V}\frac{\int_{\mathbb{R}^n}|\nabla v|^2\mathrm{d}\Ln+\beta \int_{J_\U}(\overline{v}^2+\underline{v}^2)\mathrm{d}\Hs}{\int_{\mathbb{R}^n}v^2\mathrm{d}\Ln}=\Lambda_i\Big[A(\U)^{-\frac{1}{2}}B(\U)A(\U)^{-\frac{1}{2}}\Big],
\end{equation}
where $\Lambda_i$ designates the $i$-th eigenvalue of a symmetric matrix.\\
We will say that $\U$ is normalized if $A(\U)=I_k$ and $B(\U)$ is the diagonal $(\lambda_1(\U;\beta),\hdots,\lambda_k(\U;\beta))$. Following the spectral theorem, for any linearly independant $\U\in\Uk$ there exists $P\in \text{GL}_k(\R)$ such that $P\U$ is normalized.\\
Although we expect the optimal sets to have rectifiable boundary, we may define the eigenvalues with Robin boundary conditions for any open set $\Om\subset \Rn$ as
\begin{equation}\label{defopen}
\lambda_k(\Om;\beta):=\inf\Big[\lambda_k(\U;\beta), \U\in \Uk\text{ linearly independant}:\Hs(J_\U\setminus \partial\Om)=\Ln(\left\{\U\neq 0\right\}\setminus \Om)=0\Big].
\end{equation}
\end{definition}

It may be checked that for any bounded Lipschitz domain, the admissible space corresponds to linearly independant functions $\U\in H^1(\Om)^k$ so this definition is coherent with the usual.

\section{Strictly monotonous functionals}

Let us first restate the first main result in the $\SBV$ framework. We define the admissible set of functions as
\[\Uk(m)=\left\{ \V\in\Uk:\ \V\text{ is linearly independant and }|\{\V\neq 0\}|=m\right\}.\]

For any linearly independant $\U\in\Uk$, we let:
\begin{align*}
\F(\U)&:=F(\lambda_1(\U;\beta),\hdots,\lambda_k(\U;\beta)),\\
\F_\gamma(\U)&:=\F(\U)+\gamma|\left\{\U\neq 0\right\}|.
\end{align*}
Our goal is now to show that $\F$ has a minimizer in $\Uk(m)$, and that any minimizer of $\F$ in $\Uk(m)$ is deduced from an open set, meaning there is an open set $\Om$ that essentially contains $\left\{\U\neq 0\right\}$ such that $\U_{|\Om}\in H^1(\Om)^k$. This is not the case for every $\SBV$ functions: some may have a dense and non-closed jump set, while $\partial\Om$ is closed and not dense.\\
The lemma \ref{PenalizationLemma} will make a link between minimizers of $\F$ in $\Uk(m)$ and minimizers of $\F_\gamma$ among linearly independant $\U\in\Uk$ for which the support's measure is less than $m$.

\subsection{A priori estimates}

An internal relaxed minimizer of $\F_\gamma$ is a linearly independant function $\U\in\Uk$ such that for any linearly independant $\V\in \Uk$ verifying $|\left\{\V\neq 0\right\}\setminus\left\{\U\neq 0\right\}|=0$:
\[\F_\gamma(\U)\leq \F_\gamma(\V).\]

To shorten some notations, we introduce the function $G:S_k^{++}(\R)\to \R$ such that
\[\F_\gamma(\U)=G\Big[A(\U)^{-\frac{1}{2}}B(\U)A(\U)^{-\frac{1}{2}}\Big]+\gamma|\left\{\U\neq 0\right\}|,\]
meaning that for any positive definite symmetric matrix $S$, $G\Big[S\Big]=F\left(\Lambda_1\Big[S\Big],\hdots,\Lambda_k\Big[S\Big]\right)$.
The smoothness of $F$ does not imply the smoothness of $G$ in general, because of the multiplicities of eigenvalues. However the monotonicity of $F$ implies the monotonicity of $G$ in the following sense: suppose $M,N$ are positive symmetric matrices, then:
\begin{align*}
G\Big[M+N\Big]&\leq G\Big[M\Big]+\left(\max_{i=1,\hdots,k}\sup_{\Lambda_j\Big[M\Big]\leq\lambda_j\leq \lambda_j(M+N)}\frac{\partial F}{\partial^\pm\lambda_i}(\lambda_1,\hdots,\lambda_k)\right) \Tr\Big[N\Big],\\
G\Big[M+N\Big]&\geq G\Big[M\Big]+\left(\min_{i=1,\hdots,k}\inf_{\lambda_j\Big[M\Big]\leq\lambda_j\leq \lambda_j\Big[M+N\Big]}\frac{\partial F}{\partial^\pm\lambda_i}(\lambda_1,\hdots,\lambda_k)\right) \Tr\Big[N\Big].\end{align*}
Above $\frac{\partial F}{\partial^\pm\lambda_i}$ designates the directional partial derivatives of $F$.
Moreover, $G$ has directional derivative everywhere; let $M=\begin{pmatrix}\lambda_1 \\ & \ddots \\ & & \lambda_k\end{pmatrix}$ be a diagonal matrix with $p$ distincts eigenvalues and $1\leq i_1<i_2<i_p\leq k$ be such that for any $i\in I_l:=[i_l,i_{l+1})$:
\[\lambda_{i_l}=\lambda_i<\lambda_{i_{l+1}}.\]
Then for each $i\in I_l$ the function $N\mapsto \Lambda_i\Big[N\Big]$ admits the following directional derivative at $M$:
\[\Lambda_i\Big[M+N\Big]=\Lambda_i\Big[M\Big]+\Lambda_{i-i_l+1}\Big[N_{|I_l}\Big]+\underset{N\to 0}{o}(N),\]
where $N_{|I}:=(N_{i,j})_{i,j\in I}$. Since $F$ has a directional derivative everywhere, this means that $G$ admits a directional derivative

\begin{equation}\label{Diff}
G\Big[M+N\Big]=G\Big[M\Big]+F_0\left(\Lambda_1\Big[ N_{|I_1}\Big],\hdots,\Lambda_{k-i_{k}+1}\Big[ N_{|I_p}\Big]\right)+\underset{N\to 0}{o}\left(N\right),
\end{equation}

where $F_0$ is a positiverly homogeneous function that is the directional derivative of $F$ at $(\lambda_1,\hdots,\lambda_k)$.

\begin{proposition}\label{apriori}
Let $\U$ be a relaxed internal minimizer of $\F_\gamma$, suppose it is normalized. Then there exists constants $M,\delta,R>0$ that only depend on $(n,k,\beta,\F_\gamma(\U),F)$ such that
\[\delta 1_{\left\{ \U\neq 0\right\}}\leq |\U|\leq M.\]
Moreover, up to translation of its connected component, $\U$ is supported in a set of diameter bounded by $R$.
\end{proposition}
Estimates of the form $|\U|\geq \delta 1_{\{\U\neq 0\}}$ for solution of elliptic equations with Robin boundary conditions appear in \cite{BL14}, \cite{BG15}, \cite{CK16}, see also \cite{BMV21} in a context without free discontinuity. It is a crucial steps to show the regularity of the function $\U$; once $\U$ is known to take values between two positive bounds, then it may be seen as a quasi-minimizer of the Mumford-Shah functional $\int_{\Rn}|\nabla\U|^2\mathrm{d}\Ln+\Hs(J_\U)$ on which the techniques used to show the regularity of Mumford-Shah minimizers (see \cite{DCL89}) may be extended (see \cite{CK16}, \cite{BL14}).
\begin{proof}
We show, in order, that the eigenvalues $\lambda_i(\U;\beta)$ are bounded above and below, the $L^\infty$ bound on $\U$, the lower bound on $\U_{|\left\{\U\neq 0\right\}}$, a lower bound on the Lebesgue density of $\left\{\U\neq 0\right\}$, and then the boundedness of the support.
\begin{itemize}[label=\textbullet]
\item Since $|\left\{\U\neq 0\right\}|\leq  \F_\gamma(\U)/\gamma$, then by the Faber-Krahn inequality with Robin Boundary conditions (as proved in \cite{BG10}) $\lambda_1(\U;\beta)\geq \lambda_1(\B^{|\F_\gamma(\U)|/\gamma};\beta)=:\lambda$.\\
In a similar way, since
\[F(\lambda,\hdots,\lambda,\lambda_k(\U;\beta))\leq F(\lambda_1(\U;\beta),\hdots,\lambda_k(\U;\beta))\leq \F_\gamma(\U)\]
and $F$ diverges when its last coordinate does, so $\lambda_k(\U;\beta)$ is bounded by a constant $\Lambda>0$ that only depends on the behaviour of $F$ and $\F_\gamma(\U)$. Let us write:
\begin{align*}
a&=\inf_{\frac{1}{2}\lambda\leq \lambda_1\leq \lambda_2\leq\hdots\leq \lambda_k\leq 2\Lambda}\inf_{i=1,\hdots,k}\frac{\partial F}{\partial^\pm \lambda_i}(\lambda_1,\hdots,\lambda_k),\\
b&=\sup_{\frac{1}{2}\lambda\leq \lambda_1\leq \lambda_2\leq\hdots\leq \lambda_k\leq 2\Lambda}\sup_{i=1,\hdots,k}\frac{\partial F}{\partial^\pm \lambda_i}(\lambda_1,\hdots,\lambda_k).\end{align*}
$a$ and $b$ are positive and only depend on $\F_\gamma(\U)$ and the behaviour of $F$.
\item For the $L^\infty$ bound we use a Moser iteration procedure (see for instance \cite[Th 4.1]{HL11} for a similar method). We begin by establishing that $u_i$ is an eigenfunction of $\lambda_i(\U;\beta)$ in a variational sense. \\
Let $v_i\in \mathcal{U}_1$ be such that $\left\{v_i\neq 0\right\}\subset \left\{\U\neq 0\right\}$ and $J_{v_i}\subset J_\U$, we show that $V(u_i,v_i)=0$, where
\[V(u_i,v_i):=\int_{\Rn}\nabla u_i\cdot\nabla v_i\mathrm{d}\Ln+\beta\int_{J_\U}u_i v_i\mathrm{d}\Hs-\lambda_i(\U;\beta)\int_{\R^n}u_i v_i \mathrm{d}\Ln.\]
For this consider $\U_t=\U-t (v_i-\sum_{j\neq i}V(v_i,u_j)u_j) e_i$. Since $A(\U_t)$ converges to $I_k$, $\U_t$ is linearly independant for a small enough $t$ and
\[\F_\gamma(\U)\leq \F_\gamma(\U_t).\]
This implies, since $\left\{\U_t\neq 0\right\}\subset\left\{\U\neq 0\right\}$, that
\[F(\lambda_1(\U;\beta),\hdots,\lambda_k(\U;\beta))\leq F(\lambda_1(\U_t;\beta),\hdots,\lambda_k(\U_t;\beta)),\]
which may also be written
\[G\Big[B(\U)\Big]\leq G\Big[A(\U_t)^{-\frac{1}{2}}B(\U_t)A(\U_t)^{-\frac{1}{2}}\Big].\]
Now, $A(\U_t)^{-\frac{1}{2}}B(\U_t)A(\U_t)^{-\frac{1}{2}}=B(\U)-(e_i e_i^*)V(u_i,v_i)t+\mathcal{O}(t^2)$. Suppose that $V(u_i,v_i)>0$. Let $i'$ be the lowest index such that $\lambda_{i'}(\U;\beta)=\lambda_i(\U;\beta)$. Then knowing the directional derivative of $G$ given in \eqref{Diff} we obtain (for $t>0$)
\[G\Big[A(\U_t)^{-\frac{1}{2}}B(\U_t)A(\U_t)^{-\frac{1}{2}}\Big]=G\Big[B(\U)\Big]+tV(u_i,v_i)F_0\left(0,0,\hdots,0,-1,0,\hdots,0\right)+\underset{t\to 0}{o}(t),\]
which is less than $G\Big[B(\U)\Big]$ for a small enough $t$: this is a contradiction. When $V(u_i,v_i)\leq 0$ we may do the same by replacing $v_i$ with $-v_i$. Thus for all $v_i$ with support and jump set included in the support and jump set of $\U$

\[\int_{\Rn}\nabla u_i\cdot\nabla v_i\mathrm{d}\Ln+\beta\int_{J_\U}u_i v_i\mathrm{d}\Hs=\lambda_i(\U;\beta)\int_{\R^n}u_i v_i \mathrm{d}\Ln.\]

Now we use Moser iteration methods. Let $\alpha\geq 2$ be such that $u_i\in L^\alpha$, then by taking $v_i$ to be a truncation of $|u_i|^{\alpha-2}u_i $ in $[-M,M]$ for $M\to\infty$ in the variational equation above, we obtain

\[\int_{\Rn}(\alpha-1) |u_i|^{\alpha -2}|\nabla u_i|^2\mathrm{d}\Ln+\int_{J_\U}(|\overline{u_i}|^{\alpha}+|\underline{u_i}|^{\alpha})\mathrm{d}\Hs=\lambda_i(\U;\beta)\int_{\Rn}|u_i|^\alpha \mathrm{d}\Ln.\]
Using the embedding $BV(\Rn)\hookrightarrow L^{\frac{n}{n-1}}(\Rn)$ we have:
\begin{align*}
\Vert u_i^\alpha\Vert_{L^{\frac{n}{n-1}}}&\leq C_n\Vert u_i^\alpha\Vert_{BV}\\
&\leq C_n\left(\int_{\Rn} |\nabla (|u_i|^{\alpha -1}u_i)|\mathrm{d}\Ln+\int_{J_\U}(|\overline{u_i}|^{\alpha}+|\underline{u_i}|^{\alpha})\mathrm{d}\Hs\right)\\
&\leq C_n\left(\int_{\Rn}\alpha\left(|u_i|^\alpha+|u_i|^{\alpha-2}|\nabla u_i|^2\right)\mathrm{d}\Ln+\int_{J_\U}(|\overline{u_i}|^{\alpha}+|\underline{u_i}|^{\alpha})\mathrm{d}\Hs\right)\\
&\leq C_{n,\beta}\left(\alpha+\lambda_i(\U;\beta)\right)\Vert u_i\Vert_{L^\alpha}^{\alpha}.\\
\end{align*}
And so $\Vert u_i\Vert_{L^{\frac{n}{n-1}\alpha}}\leq \Big[ C_{n,\beta}\left(\alpha+\lambda_i(\U;\beta)\right)\Big]^\frac{1}{\alpha}\Vert u_i\Vert_{L^\alpha}$. We may apply this iteratively with $\alpha_p=2\left(\frac{n}{n-1}\right)^p$ to obtain an $L^\infty$ bound of $u_i$ that only depends on $n,\beta$ and $\lambda_i(\U;\beta)$. In fact using the Faber-Krahn inequality for Robin conditions $\lambda_i(\U;\beta)\geq \lambda_1\left(\B^{|\{\U\neq 0\}|};\beta\right)$ the previous inequality applied to $\alpha_p$ may be simplified into
\[\log\left(\frac{\Vert u_i\Vert_{L^{\alpha_{p+1}}}}{\Vert u_i\Vert_{L^{\alpha_{p}}}}\right)\leq \left(C(n,\beta,|\{\U\neq 0\}|)(p+1)+\frac{1}{2}\log\lambda_i(\U;\beta)\right)\left(\frac{n-1}{n}\right)^p,\]
and summing in $p$ we obtain an estimate of the form $\Vert u_i\Vert_{L^\infty}\leq C(n,\beta,|\{\U\neq 0\}|)\lambda_i(\U;\beta)^\frac{n}{2}$.
\item Lower bound on $\U$: our goal is first to obtain an estimate of the form
\begin{equation}\label{EstCaff}
\Tr\Big[B_t\Big]+|\left\{ 0<|\U|\leq t\right\}|\leq \frac{1}{\eps}\Tr\Big[\beta_t\Big],
\end{equation}
where $\eps>0$ is a constant that only depends on the parameters and
\begin{align*}
(B_t)_{i,j}&=\int_{|\U|\leq t}\nabla u_i\cdot\nabla u_j\mathrm{d}\Ln +\beta\int_{J_\U}\left(\overline{u_i 1_{|\U|\leq t}}\cdot\overline{u_j 1_{|\U|\leq t}}+\underline{u_i 1_{|\U|\leq t}}\cdot\underline{u_j 1_{|\U|\leq t}}\right)\mathrm{d}\Hs,\\
(\beta_t)_{i,j}&=\beta\int_{\partial^*\left\{ |\U|>t\right\}\setminus J_\U}u_i u_j\mathrm{d}\Hs.
\end{align*}
This is intuitively what we obtain by comparing $\U$ and $\U1_{\left\{|\U|>t\right\}}$. From this we will derive a lower bound of $\inf_{\U\neq 0}|\U|$ with similar arguments as what was done in \cite{CK16}. Suppose \eqref{EstCaff} does not hold. This means that, since $B_t\leq B$ and $|\left\{\U\neq 0\right\}|\leq \F_\gamma(\U)$,
\[\beta_t\leq \left(B(\U)+\F_\gamma(\U)I_k\right)k\eps\leq c\eps B(\U),\]
for a certain $c>0$ since $B(\U)\geq \lambda I_k$. Let us now compare $\U$ with $\U_t=\U1_{\left\{ |\U|>t\right\}}$; this function is admissible for a small enough $t$ because $A(\U_t)=I_k-A\left(\U 1_{\left\{ 0<|\U|\leq t\right\} }\right)$, so
\[\Vert A(\U_t)-I_k\Vert\leq Ct^2|\left\{ 0<|\U|\leq t\right\}|.\]
Notice also that $B(\U_t)=B(\U)-B_t+\beta_t$. Then the optimality condition $\F_\gamma(\U)\leq \F_\gamma(\U_t)$ gives
\begin{equation}\label{eqInter}
G\Big[B(\U)\Big]+\gamma|\left\{0<|\U|\leq t\right\}|\leq G\Big[A(\U_t)^{-\frac{1}{2}}(B(\U)-B_t+\beta_t)A(\U_t)^{-\frac{1}{2}}\Big].
\end{equation}
We first show that $B_t$ is small enough for small $t$. With our hypothesis on $\beta_t$ and the fact that $A(\U_t)^{-\frac{1}{2}}\leq I_k+Ct^2I_k\leq (1+c\eps)I_k$ for a small enough $t$
\[G\Big[B(\U)\Big]+\gamma|\left\{0<|\U|\leq t\right\}|\leq G\Big[\Big[1+2c\eps\Big]B(\U)-B_t\Big].\]
So $G\Big[B(\U)\Big]\leq G\Big[(1+2c\eps)B(\U)-B_t\Big]$. Now, there exists $i\in\left\{1,\hdots,k\right\}$ such that
\[\Lambda_i\Big[(1+c\eps)B(\U)-B_t\Big]\leq (1+c\eps)\Lambda_i\Big[B(\U)\Big]-\frac{1}{k}\Tr\Big[B_t\Big].\]
And so, using the monotonicity of $F$ and the definition of $a,b$ in the first part of the proof:
\begin{align*}
G\Big[B(\U)\Big]&\leq G\Big[(1+c\eps)B(\U)-B_t\Big]\\
&\leq F\left((1+c\eps)\lambda_1(\U;\beta),\hdots,(1+c\eps)\lambda_{i-1}(\U;\beta),(1+c\eps)\lambda_i(\U;\beta)-\frac{1}{k}\Tr\Big[B_t\Big],\hdots,(1+c\eps)\lambda_k(\U;\beta)\right)\\
&\leq G\Big[B(\U)\Big]+bc\eps\Tr\Big[B(\U)\Big]-a\min\left(\frac{1}{k}\Tr\Big[B_t\Big],\frac{\lambda}{2}\right).
\end{align*}

With a small enough $\eps$, we obtain $\Tr\Big[B_t\Big]\leq \frac{\lambda}{2}$. Now we may come back to \eqref{eqInter}, and using the fact that $A(\U_t)^{-\frac{1}{2}}\leq (1+Ct^2|\left\{0<|\U|\leq t\right\}|)I_k$ we obtain
\[G\Big[B(\U)\Big]+\gamma|\left\{0<|\U|\leq t\right\}|\leq G\Big[B(\U)+\beta_t-B_t+Ct^2|\left\{0<|\U|\leq t\right\}|I_k\Big],\]
and so with the monotonicity of $G$
\[G\Big[B(\U)\Big]+\gamma|\left\{0<|\U|\leq t\right\}|\leq G\Big[B(\U)\Big] +b\Tr\Big[\beta_t\Big]-a\Tr\Big[B_t\Big]+Cbt^2|\left\{0<|\U|\leq t\right\}|.\]
In particular, for a small enough $t>0$ (depending only on the parameters)
\[a\Tr\Big[B_t\Big]+\frac{\gamma}{2}|\left\{0<|\U|\leq t\right\}|\leq b\Tr\Big[\beta_t\Big],\]
and so we obtained that there is a big enough constant $C>0$, and a small enough $t_1>0$, such that for any $t\in (0,t_1]$
\begin{equation}\label{EstCaff2}
\Tr\Big[B_t\Big]+|\left\{0<|\U|\leq t\right\}|\leq C\Tr\Big[\beta_t\Big].
\end{equation}
Now we let
\[V:=|\U|=\sqrt{u_1^2+\hdots+u_k^2}\left(\geq \delta 1_{\left\{ \U\neq 0\right\}}\right).\]
Let $f(t)=\int_0^t \tau\Hs(\partial^*\left\{ V> t\right\}\setminus J_\U)\mathrm{d}\tau$. Notice that the right-hand side of \eqref{EstCaff2} is $Ctf'(t)$. Then for any $t\leq t_1$:
\begin{align*}
f(t)&\leq \int_0^t \tau\Hs(\partial^*\left\{ V> \tau\right\}\setminus J_V)d\tau=\int_{\om_t}V|\nabla V|\mathrm{d}\Ln\\
&\leq |\left\{ 0<V\leq t\right\}|^{\frac{1}{2n}}\left(\int_{\left\{ 0<V\leq t\right\}}|\nabla V|^2\mathrm{d}\Ln\right)^{\frac{1}{2}}\left(\int_{\left\{ 0<V\leq t\right\}}(V^2)^{\frac{n}{n-1}}\mathrm{d}\Ln\right)^\frac{n-1}{2n}\\
&\leq C(tf'(t))^{\frac{1}{2n}+\frac{1}{2}}\left(|D(V^2)|(\left\{ 0<V\leq t\right\})\right)^\frac{1}{2}\\
&\leq C(tf'(t))^{\frac{1}{2n}+\frac{1}{2}}\left(\int_{\left\{ 0<V\leq t\right\}}V|\nabla V|\mathrm{d}\Ln +\int_{J_V\cap \left\{ 0<V\leq t\right\}}V^2\mathrm{d}\Hs\right)^\frac{1}{2}\\
&\leq C(tf'(t))^{\frac{1}{2n}+\frac{1}{2}}\left(t|\left\{ 0<V\leq t\right\}|^{\frac{1}{2}}\left(\int_{\left\{ 0<V\leq t\right\}}|\nabla V|^2\mathrm{d}\Ln\right)^{\frac{1}{2}} +\int_{J_V\cap \left\{ 0<V\leq t\right\}}V^2\mathrm{d}\Hs\right)^\frac{1}{2}\\
&\leq C(tf'(t))^{1+\frac{1}{2n}}.
\end{align*}
The constant $C>0$ above depends only on the parameters and may change from line to line. This implies that $f'(t)f(t)^{-\frac{2n}{2n+1}}\geq ct^{-1}$, so for any $t\in ]0,t_1[$ such that $f(t)>0$ this may be integrated from $t$ to $t_1$ to obtain
\[\frac{1}{2n+1}f(t_1)^{\frac{1}{2n+1}}\geq c\log(t_1/t).\]
Since $f(t_1)\leq |\left\{\U\neq 0\right\}|^{\frac{1}{2}}\left(\int_{\Rn}|\nabla V|^2\right)^\frac{1}{2}\leq \sqrt{k\F_\gamma(\U)\Lambda/\gamma}$, then $t$ is bounded below in terms of the parameters of the problem. This means that $f(\delta)=0$ for a certain explicit $\delta>0$. In particular, $(\nabla\U)1_{\left\{ |\U|\leq \delta\right\}}=0$, and by comparing $\U$ with $\U_\delta$, \eqref{EstCaff} becomes $|\left\{ 0<|\U|\leq \delta\right\}|\leq 0$, so we obtained
\[|\U|\geq \delta 1_{\left\{ \U\neq 0\right\}}.\]

\item To show the support $\left\{\U\neq 0\right\}$ (or its connected components) is bounded, we begin by showing a lower estimate for the Lebesgue density on this set. This is obtained by comparing $\U$ with $\U_r:=u1_{\mathbb{R}^n\setminus \B_r}$ where $\B_r$ is a ball of radius $r>0$.\bigbreak

As previously, we first need to check that $\U_r$ is admissible for any small enough $r>0$. With the $L^\infty$ bound on $\U$, we get $|A(\U_r)-I_k|\leq C|\left\{ \U\neq 0\right\}\cap \B_r|$. In particular, $|A(\U_r)-I_k|\leq C r^n$, which proves that $A(\U_r)$ is invertible for a small enough $r$.\bigbreak
Let $f(r)=|\left\{ \U\neq 0\right\}\cap \B_r|$. By comparing $\U$ with $\U_r$ we obtain
\[G\Big[B(\U)\Big]+\gamma|\left\{\U\neq 0\right\}\cap \B_r|\leq G\Big[A_r(B-B_r+\beta_r)A_r\Big],\]
where $A_r,B_r,\beta_r$ are defined as previously: $A_r=A(\U_r)^{-\frac{1}{2}}$ and

\begin{align*}
(B_r)_{i,j}&=\int_{\B_r}\nabla u_i\cdot\nabla u_j \mathrm{d}\Ln+\beta\int_{J_\U}\left(\overline{u_i 1_{\B_r}}\cdot\overline{u_j 1_{\B_r}}+\underline{u_i 1_{\B_r}}\cdot\underline{u_j 1_{\B_r}}\right)\mathrm{d}\Hs,\\
(\beta_r)_{i,j}&=\beta\int_{\partial\B_r\setminus J_\U}u_i u_j\mathrm{d}\Hs.
\end{align*}
With the same argument as what we did to obtain the lower bound, this estimate implies that for any $r\in ]0,r_0]$ where $r_0$ is small enough

\[c\Tr\Big[B_r\Big]\leq \Tr\Big[\beta_r\Big]+f(r),\]
for a certain $c>0$. With the $L^\infty$ bound and the lower bound on $\U$, we deduce that for a certain constant $C>0$:
\[\Hs(\B_r\cap J_\U)\leq C\left(f(r)+\Hs(\partial \B_r\cap \left\{ \U\neq 0\right\})\right).\]

Notice that $f'(r)=\Hs(\partial \B_r\cap \left\{ \U\neq 0\right\})$, so with the isoperimetric inequality
\[c_n f(r)^{1-\frac{1}{n}}\leq \Hs(\B_r\cap J_\U)+\Hs(\partial \B_r\cap \left\{ \U\neq 0\right\})\leq C(f(r)+f'(r)).\]
Since $f(r)\leq Cr^n\rightarrow 0$, we deduce that for a certain constant $C>0$ and any small enough $r$ ($r<r_0$) we have
\[f(r)^{1-\frac{1}{n}}\leq Cf'(r).\]
Suppose now that $f(r)>0$ for any $r>0$. Then by integrating the above estimate from $0$ to $r_0$, we obtain that for a certain constant $c>0$ and any $r\in [0,r_0]$
\[|\left\{ \U\neq 0\right\}\cap \B_{x,r}|\geq cr^n.\]
Consider now a system of points $S\subset\Rn$ such that for any $x\in S$ and any $r>0$, $|\left\{\U\neq 0\right\}\cap\B_{x,r}|>0$, and such that for any distinct $x,y\in S$, $|x-y|\geq 2r_0$. Then
\[\F_\gamma(\U)\geq \gamma |\left\{ \U\neq 0\right\}|\geq \gamma\sum_{x\in S}|\left\{ \U\neq 0\right\}\cap \B_{x,r_0}|\geq c\gamma r_0^n \mathrm{Card}(S),\]
so $\mathrm{Card}(S)$ is bounded. Then by taking a maximal set of separated points $S$ as above, the balls $(\B_{x,2r_0})_{x\in S}$ cover $\left\{\U\neq 0\right\}$. This means in particular that the support of $u$ is bounded by a constant only depending on the parameters, up to a translation of the its connected components.

\end{itemize}
\end{proof}

\subsection{Existence of a relaxed minimizer with prescribed measure}

This section is dedicated to the proof of the following result.

\begin{proposition}\label{existence}
Let $m,\beta>0$, then there exists $\U\in \Uk$ that minimizes $\F$ in the admissible set $\Uk(m)$.
\end{proposition}
We begin with a lemma that will help us to show that any minimizing sequence of $\F$ in $\Uk(m)$ has concentration points, meaning points around which the measure of the support is bounded below by a positive constant.
\begin{lemma}\label{Conc}
Let $\U\in \Uk$, we let $K_p:=p+[-\frac{1}{2},\frac{1}{2}]^n$, then there exists $p\in \mathbb{Z}^n$ such that
\[|\{\U\neq 0\}\cap K_p|\geq \left(\frac{c_n\Vert \U\Vert_{L^2(\Rn)}^2}{\Vert \U\Vert_{L^2(\Rn)}^2+\int_{\Rn}|\nabla\U|^2\mathrm{d}\Ln+\int_{J_\U}\left(|\overline{\U}|^2+|\underline{\U}|^2\right)\mathrm{d}\Hs}\right)^n.\]
\end{lemma}
\begin{proof}
It is the consequence of the $BV(K_p)\hookrightarrow L^\frac{n}{n-1}(K_p)$ embedding, see \cite[lemma 12]{BGN21}.
\end{proof}

The following lemma makes a straightforward link between minimizers of $\F$ with fixed volume and interior minimizers of $\F_\gamma$ for a sufficiently small $\gamma$, which means that all the a priori estimates apply.
\begin{lemma}\label{PenalizationLemma}
Let $\U\in \Uk$ be a minimizer of $\F$ in the admissible set
\[\left\{ \V\in\Uk:\ \V\text{ is linearly independant and }|\{\V\neq 0\}|=m\right\}.\]
Then there exists $\gamma>0$ depending only on $(n,m,\beta,\F(u),F)$ such that $\U$ is a minimizer of $\F_\gamma$ in the admissible set
\[\left\{ \V\in\Uk:\ \V\text{ is linearly independant and }|\{\V\neq 0\}|\in ]0,m]\right\}.\]
\end{lemma}
\begin{proof}
Consider a linearly independant $\V\in\Uk$ such that $\delta:=\frac{|\{\V\neq 0\}|}{|\{\U\neq 0\}|}\in ]0,1]$. Let $\W(x):=\V(x\delta^{1/n})$. Then the support of $\W$ has the same measure as $\U$ and so $\F(\U)\leq \F(\W)$. Looking how the matrices $A$ and $B$ scale with the change of variable $x\to x\delta^{-\frac{1}{n}}$ we obtain
\[A(\W)^{-\frac{1}{2}}B(\W)A(\W)^{-\frac{1}{2}}\leq \delta^{\frac{1}{n}}A(\V)^{-\frac{1}{2}}B(\V)A(\V)^{-\frac{1}{2}},\]
hence
\[\F(\U)\leq F\left(\delta^{\frac{1}{n}}\lambda_1(\V;\beta),\hdots,\delta^{\frac{1}{n}}\lambda_k(\V;\beta)\right).\]
By the Faber-Krahn inequality for Robin eigenvalues, $\lambda_1(\V;\beta)\geq \lambda_1(\B^{|\{\V\neq 0\}|};\beta)$. Moreover since $F$ diverges when its last coordinate does, we may suppose without loss of generality that $\lambda_k(\V;\beta)$ is bounded by a certain constant $\Lambda>0$ that does no depend on $\V$. This in turn means that $|\{\V\neq 0\}|$ is bounded below by a positive constant depending only on $n,\beta,\Lambda$ by the Faber-Krahn inequality, so $\delta$ is bounded below. Then by denoting $a$ the minimum of the partial derivatives of $F$ on $[\delta^{\frac{1}{n}}\lambda_1(\B^{m};\beta),\Lambda]^{k}$, we obtain
\[\F(\U)\leq \F(\V)-a(\lambda_1(\V;\beta)+\hdots+\lambda_k(\V;\beta))(1-\delta^{1/n})\leq \F(\V)-\frac{ka\lambda_1(\B^m;\beta)}{nm}(|\{\U\neq 0\}|-|\{\V\neq 0\}|).\]
This concludes the proof.
\end{proof}

We may now prove the main result of this section.
\begin{proof}
We proceed by induction on $k$. The main idea is that we either obtain the existence of a minimizer by taking the limit of a minimizing sequence, or we don't and in this case the minimizer is disconnected so it is the union of two minimizers of different functionals depending on strictly less than $k$ eigenvalues.\\
The initialisation for $k=1$ amounts to showing there is a minimizer for $\lambda_1(\U;\beta)$ in $\mathcal{U}_1(m)$: this has been done in \cite{BG15} and it is known to be the first eigenfunction of a ball of measure $m$.\\
Suppose now that $k\geq 2$ and the result is true up to $k-1$. Consider $(\U^i)_i$ a minimizing sequence for $\F$ in $\Uk(m)$. Then the concentration lemma \ref{Conc} may be applied to each $\U^i$ to find a sequence $(p^i)_i$ in $\mathbb{Z}^n$ such that
\begin{equation}\label{EstVol}
\liminf_{i\to\infty}|K_{p_i}\cap \{\U^i\neq 0\}|>0.
\end{equation}
We lose no generality in supposing, up to a translation of each $\U^i$, that $p^i=0$. Now with the compactness lemma \ref{CompactnessLemma}, we now up to extraction that $\U^i$ converges in $L^2_\text{loc}$ to a certain function $\U\in \Uk$ with local lower semicontinuity of its Dirichlet-Robin energy.\\
We now split $\U^i$ into a "local" part and a "distant" part; we may find an increasing sequence $R^i\to\infty$ such that
\[\U^i 1_{\B_{p^i,R^i}}\underset{L^2(\Rn)}{\longrightarrow} \U.\]
Up to changing each $R^i$ with a certain $\tilde{R^i}\in [\frac{1}{2}R^i,R^i]$, we may suppose that
\[\int_{\partial \B_{R^i}\setminus J_{\U^i}}|\U^i|^2\mathrm{d}\Hs=\underset{i\to \infty}{o}(1),\]
so that for each $i\in \{1,\hdots,k\}$
\[\lambda_i\left((\U1_{\B_{R^i}},\U1_{\B_{R^i}^c});\beta\right)\leq\lambda_i(\U;\beta)+\underset{i\to \infty}{o}(1).\]
Since $A\left(\U1_{\B_{R^i}},\U1_{\B_{R^i}^c}\right)$ and $B\left(\U1_{\B_{R^i}},\U1_{\B_{R^i}^c}\right)$ are block diagonal (with two blocks of size $k\times k$), then up to extraction on $i$ there is a certain $p\in \{0,1,\hdots,k\}$ such that
\[\Big[\lambda_1(\U^i 1_{\B_{R^i}};\beta),\hdots,\lambda_p(\U^i 1_{\B_{R^i}};\beta),\lambda_1(\U^i 1_{\B_{R^i}^c};\beta),\hdots,\lambda_p(\U^i 1_{\B_{R^i}^c};\beta)\Big]^{\mathfrak{S}_k}\leq (\lambda_1(\U^i;\beta),\hdots,\lambda_k(\U^i;\beta))+\underset{i\to \infty}{o}(1),\]
where $\Big[a_1,\hdots,a_k\Big]^{\mathfrak{S}_k}$ designate the ordered list of the values $(a_1,\hdots,a_k)$. There are now three cases:
\begin{itemize}[label=\textbullet]
\item $p=0$: we claim this can not occur. Indeed this would mean that $\U^i1_{\B_{R^i}^c}$ is such that
\[\F(\U^i1_{\B_{R^i}^c})\underset{i\to\infty}{\longrightarrow}\inf_{\Uk(m)}\F.\]
However, because of \eqref{EstVol} we know there is a certain $\delta>0$ such that for all big enough $i$ the measure of the support of $\U^i1_{\B_{R^i}^c}$ is less than $m-\delta$. Letting $\V^i=\U^i1_{\B_{R^i}^c}\left(\Big[\frac{m-\delta}{m}\Big]^\frac{1}{n}\cdot\right)$, $\V^i$ is a linearly independant sequence of $\Uk$, with support of volume less than $m$, such that $\F(\V^i)<\inf_{\Uk(m)}\F$ for a big enough $i$: this is a contradiction.
\item $p=k$. In this case $\U(=\lim_i \U^i1_{\B_{R^i}})$ is a minimizer of $\F$ with measure less than $m$. This is because, in addition to the fact that $\U^i1_{\B_{R^i}}$ converges to $\U$ in $L^2$, the lower semi-continuity result tells us that for each $z\in\R^k$:
\[z^*B(\U)z\leq \liminf_{i}z^*B(\U^i1_{\B_{R^i}})z,\]
thus for any $j=1,\hdots,k$, $\lambda_j(\U;\beta)\leq \liminf_{i}\lambda_j(\U^i1_{\B_{R^i}};\beta)$. And $|\{\U\neq 0\}|\leq \liminf |\{\U^i1_{B_{R^i}}\neq 0\}|\leq m$.
\item $1\leq p\leq k-1$. This is where we will use the induction hypothesis. We let:
\begin{align*}
\lambda_j&=\lim_{i\to\infty}\lambda_j(\U^i1_{\B_{R^i}};\beta),\ \forall j=1,\hdots,p & m_\text{loc}=\lim_{i\to\infty}|\{\U^i1_{B_{R^i}}\neq 0\}|,\\
\mu_j&=\lim_{i\to\infty}\lambda_j(\U^i1_{\B_{R^i}}^c;\beta),\ \forall j=1,\hdots,k-p & m_\text{dist}=\lim_{i\to\infty}|\{\U^i1_{B_{R^i}^c}\neq 0\}|.\\
\end{align*}
Then by continuity of $F$
\[\inf_{\Up(m)}\F=F\left(\Big[\lambda_1,\hdots,\lambda_p,\mu_1,\hdots,\mu_{k-p}\Big]^{\mathfrak{S}_k}\right).\]

Let us introduce
\[\F_{\text{loc}}:\V\in \mathcal{U}_{p}(m_\text{loc})\mapsto F\left(\Big[\lambda_1(\V;\beta),\hdots,\lambda_p(\V;\beta),\mu_1,\hdots,\mu_{k-p}\Big]^{\mathfrak{S}_k}\right).\]
This functional verify the hypothesis \eqref{HypF}, so following the induction hypothesis we know it has a minimizer $\V$. Moreover, according to the a priori bounds, $\V$ is known to have bounded support. Since $|\{\U^i1_{\B_{R^i}\neq 0}\}|\underset{i\to\infty}{\rightarrow}m_{\text{loc}}$, then by the optimality of $\V$ we get
\[\F_\text{loc}(\V)\leq \liminf_{i\to\infty}\F_\text{loc}(\U^i1_{\B_{R^i}})=F\left(\Big[\lambda_1,\hdots,\lambda_p,\mu_1,\hdots,\mu_{k-p}\Big]^{\mathfrak{S}_k}\right)=\inf_{\Uk(m)}\F.\]

Now consider the functional
\begin{align*}
\F_{\text{dist}}:\W\in \mathcal{U}_{k-p}(m_\text{dist})&\mapsto F\left(\Big[\lambda_1(\V;\beta),\hdots,\lambda_{p}(\V;\beta),\lambda_1(\W;\beta),\hdots,\lambda_{k-p}(\W;\beta)\Big]^{\mathfrak{S}_k}\right).
\end{align*}
With the same arguments, there is a minimizer $\W$ with bounded support. By comparing $\W$ with $\U^i1_{\B_{R^i}^c}$ we obtain
\[\F_\text{dist}(\W)\leq\liminf_{i\to\infty}\F_\text{dist}(\U^i1_{\B_{R^i}^c})=\F_\text{loc}(\V)\left(\leq \inf_{\Uk(m)}\F\right).\]

Since both $\V$ and $\W$ have bounded support we may suppose up to translation that their support are a positive distance from each other. Consider $\U=\V\oplus\W$, then $\F(\U)=\F_{\text{dist}}(\W)$ so $\U$ is a minimizer of $\F$ in $\Uk(m)$.
\end{itemize}
\end{proof}

\subsection{Regularity of minimizers}

Here we show that the relaxed global minimizer $\U$ that we found in the previous section corresponds to the eigenfunctions of an open set. What this means is that there is an open set $\Om$ that contains almost all the support of $\U$ such that $\U_{|\Om}\in H^1(\Om)^k$ and $\lambda_1(\Om;\beta),\hdots,\lambda_k(\Om;\beta)$ as defined in \eqref{defopen} are reached for $u_{1|\Om},\hdots,u_{k|\Om}$ respectively (provided $\U$ is normalized). Moreover we show that this open set $\partial\Om$ is Ahlfors regular and $\Hs(\partial\Om)<\infty$.\\
The main step is to show that $J_\U$ is essentially closed, meaning $\Hs\left(\overline{J_\U}\setminus J_\U\right)=0$. This is obvious for functions $\U$ that are eigenfunctions of a smooth open set $\Om$, since $J_\U=\partial\Om$, however an $\SBV$ function could have a dense jump set.\\
This is dealt using similar methods as in \cite{DCL89}, \cite{CK16}; we show that for every point $x\in \Rn$ with sufficiently low $(n-1)$ dimensional density in $J_\U$, the energy of $\U$ decreases rapidly around that point (this is lemma \ref{DecayLemma}). This is obtained by contradiction and blow-up methods, by considering a rescaling of a sequence of function that do not verify this estimate. 
As a consequence we obtain uniform lower bound on the $(n-1)$ dimensional density of $J_\U$, which implies that it is essentially closed.
We point out that in similar problems (see \cite{BL14}), the essential closedness of the jump set is obtained using the monotonicity of $\frac{1}{r^{n-1}}\left(\int_{\B_r}|\nabla u|^2+\Hs(J_u\cap \B_r)\right)\wedge c+c'r^{\alpha}$ for some constants $c,c',\alpha>0$ (where $u$ is a scalar solution of some similar free discontinuity problem). However our optimality condition (see \eqref{EstOpt} below) does not seem to be enough to establish a similar monotonicity property, namely due to the remainder on the right-hand side and the multiplicities of eigenvalues.
\begin{proposition}\label{regularity}
Let $\U$ be a relaxed minimizer $\F_\gamma$. Then $\Hs\left(\overline{J_\U}\setminus J_\U\right)=0$ and $\Om:=\left\{ \overline{|\U|}>0\right\}\setminus \overline{J_\U}$ is an open set such that $(u_1,\hdots,u_k)$ are the first $k$ eigenfunctions of the Laplacian with Robin boundary conditions on  $\Om$.
\end{proposition}
Since the proof is very similar to what was done in \cite{CK16}, we only sketch the specific parts of the proof that concern the vectorial character of our problem.
\begin{proof}
We first establish an optimality conditions for perturbations of $\U$ on balls with small diameter. We suppose $\U$ is normalized and, using the same notations as in \eqref{Diff} for $M=B(\U)$ we denote
\begin{equation}\label{G0}
G_0\Big[N\Big]=F_0\left(\Lambda_1\Big[ N_{|I_1}\Big],\hdots,\Lambda_{k-i_{k}+1}\Big[ N_{|I_p}\Big]\right),
\end{equation}
such that:
\begin{equation}\label{Dev}
G\Big[B(\U)+N\Big]=G\Big[B(\U)\Big]+G_0\Big[N\Big]+\underset{N\to 0}{o}(N).\end{equation}

While $G_0$ is not linear (except in the particular case where $\frac{\partial F}{\partial\lambda_i}=\frac{\partial F}{\partial \lambda_j}$ for each $i,j$ such that $\lambda_i(\U;\beta)=\lambda_j(\U;\beta)$), it is positively homogeneous. We let
\[E_0\Big[N\Big]=\max\left(G_0\Big[N\Big],\Tr\Big[N\Big]\right).\]
$E_0$ is also positively homogeneous and verify that for any non-zero $S\in S_k^+(\R)$, $E_0\Big[-S\Big]<0$. We show that:
\begin{center}\textit{For any $\V\in\Uk$ that differs from $\U$ on a ball $\B_{x,r}$ where $r$ is small enough, we have}\end{center}
\begin{equation}\label{EstOpt}
E_0\Big[B(\V;\B_{x,r})-B(\U;\B_{x,r})\Big]\geq -\Lambda r^n -\delta(r)|B(\U;\B_{x,r})|.
\end{equation}
Where $\Lambda>0$, $\delta(r)\underset{r\to 0}{\rightarrow}0$, and
\[B(\W;\B_{x,r})_{i,j}:=\int_{\B_{x,r}}\nabla w_i\cdot\nabla w_j \mathrm{d}\Ln+\beta\int_{J_\W}\left(\overline{w_i 1_{\B_{x,r}}}\cdot\overline{w_j 1_{\B_{x,r}}}+\underline{w_i 1_{\B_{x,r}}}\cdot\underline{w_j 1_{\B_{x,r}}}\right)\mathrm{d}\Hs.\]
To show \eqref{EstOpt}, we may suppose that $\Tr\Big[B(\V;\B_{x,r})\Big]\leq \Tr\Big[B(\U;\B_{x,r})\Big]$ (or else it is automatically true) and that $\V$ is bounded in $L^\infty$ by the same bound as $\U$. The optimality condition of $\U$ gives
\[\F_\gamma(\U)\leq \F_\gamma(\V),\]
where the right-hand side is well defined for any small enough $r>0$ since $|A(\V)-I_k|\leq Cr^n$. This implies
\[G\Big[B(\U)\Big]\leq G\Big[(1+Cr^n)(B(\U)-B(\U;\B_{x,r})+B(\V;\B_{x,r}))\Big]+\gamma|\B_{x,r}|.\]
Thus, using the monotonicity of $G$ and the developpement \eqref{Dev} we obtain the estimate \eqref{EstOpt}. Let us now show that this estimate, along with the a priori estimate
\begin{equation}\label{Estapriori}
\delta 1_{\left\{\U\neq 0\right\}}\leq |\U|\leq M,
\end{equation}
implies the closedness of $J_\U$, following arguments of \cite{CK16} that were originally developped in \cite{DCL89} for minimizers of the Mumford-Shah functional. The crucial argument is the following decay lemma.

\begin{lemma}\label{DecayLemma}
For any small enough $\tau \in ]0,1[$, there exists $\overline{r}=\overline{r}(\tau),\eps=\eps(\tau)>0$, such that for any $x\in \Rn$, $r\in ]0,\overline{r}]$, $\W\in \Uk$ verifying the a priori estimates \eqref{Estapriori} and the optimality condition \eqref{EstOpt}
\[\left(\Hs(J_\W\cap \B_{x,r})\leq \eps r^{n-1},\ \Tr\Big[B(\W;\B_{x,r})\Big]\geq r^{n-\frac{1}{2}}\right)\text{ implies }Tr\Big[B(\W;\B_{x,\tau r})\Big]\leq \tau^{n-\frac{1}{2}}Tr\Big[B(\W;\B_{x,r})\Big].\]
\end{lemma}
\begin{proof}
The proof is sketched following the same steps as \cite{CK16}. Consider a sequence of functions $\W^i\in \Uk$ with a sequence $r_i,\eps_i\to 0$ and a certain $\tau\in ]0,1[$ that will be fixed later, such that:
\begin{align}
\Hs(J_{\W^i}\cap \B_{r_i})&=\eps_i r_i^{n-1},\\
\Tr\Big[B(\W^i;\B_{r_i})\Big]&\geq r_i^{n-\frac{1}{2}},\\ \label{Absurd}
Tr\Big[B(\W^i;\B_{x,\tau r_i})\Big]&\geq \tau^{n-\frac{1}{2}}Tr\Big[B(\W;\B_{r_i})\Big].
\end{align}
And let
\[\V^i(x)=\frac{\W^i\left(x/r_i\right)}{\sqrt{r_i^{2-n}\Tr\Big[B(\W^i;\B_{r_i})\Big]}}.\]
Then, since $\int_{\B_1}|\nabla\V^i|^2\mathrm{d}\Ln\leq 1$ and $\Hs(J_{\V^i}\cap\B_1)=\eps_i\to 0$, we know there exists some sequences $\tau_i^-<m_i<\tau_i^+$ such that the function:
$\tilde{\V^i}:=\min(\max(\V^i,\tau_i^-),\tau_i^+)$ (where the $\min$ and $\max$ are taken for each component) verifies:
\begin{align*}
\Vert \tilde{\V^i}-m_i\Vert_{L^\frac{2n}{n-2}(\B_1)}&\leq C_n\Vert \nabla\V\Vert_{L^2(\B_1)}&\left(\leq 1\right),\\
\Ln(\{\tilde{\V^i}\neq \V^i\})&\leq C_n\Hs(J_{\V^i}\cap \B_1)^\frac{n}{n-1}&\left(=C_n\eps_i^\frac{n}{n-1}\right).
\end{align*}
One may prove (using a $\mathrm{BV}$ and a $L^{\frac{2n}{n-2}}$ bound) that $\tilde{\V}^i-m_i$ converges in $L^2$ with lower semi-continuity for the Dirichlet energy to some $\V\in H^1(\B_1)$. We claim $\V$ is harmonic as a consequence of \eqref{EstOpt}: for this consider a function $\boldsymbol{\varphi}\in H^1(\B_1)^k$ that coincides with $\V$ outside a ball $\B_\rho$ for some $\rho<1$. Let $\rho'\in ]\rho,1[$, $\eta\in \mathcal{C}^\infty_{\text{compact}}(\B_{\rho'},[0,1])$ such that $\eta=1$ on $\B_\rho$ and $|\nabla\eta|\leq 2(\rho'-\rho)^{-1}$. Then we define
\begin{align*}
\boldsymbol{\varphi}^i&= (m_i+\boldsymbol{\varphi})\eta+\tilde{\V}^i(1-\eta)1_{\B_{\rho'}}+\V^i 1_{\Rn\setminus\B_{\rho'}},\\
\boldsymbol{\Phi}^i(x)&=\sqrt{r_i^{2-n}\Tr\Big[B(\W^i;\B_{r_i})\Big]}\boldsymbol{\varphi}^i(r_ix).
\end{align*}
$\boldsymbol{\Phi}^i$ coincides with $\W^i$ outside of a ball of radius $\rho' r_i$, so it may be compared to $\W^i$ using the optimality condition \eqref{EstOpt}. With the same computations as in \cite{CK16} we obtain, as $\rho\nearrow \rho'$, that
\[E_0\Big[B(\boldsymbol{\varphi};\B_{\rho'})-B(\V;\B_{\rho'})\Big]\geq 0.\]
Taking $\boldsymbol{\varphi}$ to be the harmonic extension of $\V_{|\partial \B_\rho}$ in $\B_\rho$, we find that $B(\boldsymbol{\varphi};\B_{\rho'})\leq B(\V;\B_{\rho'})$ with equality if and only if $\V$ is equal to its harmonic extension. If it is not, then \[E_0\Big[B(\boldsymbol{\varphi};\B_{\rho'})-B(\V;\B_{\rho'})\Big]< 0,\]
which contradicts the optimality. This means that the components of $\V$ are harmonic. Since $\int_{\B_1}|\nabla \V|^2\mathrm{d}\Ln\leq 1$, then $|\nabla \V|\leq \sqrt{1/|\B_{1/2}|}$ on $\B_{1/2}$, so for any $\tau< \frac{1}{2^n|\B_1|}$ we find that $\int_{\B_\tau}|\nabla u|^2\mathrm{d}\Ln<\tau^{n-\frac{1}{2}}$; this contradicts the condition \eqref{Absurd}.
\end{proof}
The decay lemma implies the existence of $r_1,\eps_1>0$ such that for any $x\in J_\U^{\text{reg}}$ and $r\in ]0,r_1[$:
\begin{equation}\label{Ahlfors}
\Hs(J_\U\cap \B_{x,r})\geq\eps_1 r^{n-1}.
\end{equation}
Suppose indeed that it is not the case for some $x\in J_{\U}$. Let $\tau_0\in ]0,1[$ be small enough to apply lemma \ref{DecayLemma}. Then for a small enough $\tau_1$,
\[\Tr\Big[B(\U:\B_{x,\tau_1 r}\Big]\leq \delta^2\eps(\tau_0) (\tau_1r)^{n-1}.\]
Indeed, either $\Tr\Big[B(\U;\B_{x, r})\Big]$ is less than $r^{n-\frac{1}{2}}$ and this is direct provided we take $r_1<\delta^4\eps(\tau_0)^2\tau_1^{2(n-1)}$, or it is not and then by application of the lemma (and using the fact that $\Tr\Big[B(\U;\B_{x,r})\Big]\leq C(\U)r^{n-1}$, which is obtained by comparing $\U$ with $\U1_{\Rn\setminus \B_{x,r}}$) we get
\[\Tr\Big[B(\U;\B_{x,\tau_1r})\Big]\leq C(\U)\tau_1^{n-\frac{1}{2}}r^{n-1}\leq \delta^2\eps(\tau_0)(\tau_1 r)^{n-1},\]
provided we choose $\tau_1\leq C(\U)^{-2}\delta^4 \eps(\tau_0)^2$ (and $\eps_1=\eps(\tau_1),r_1<\overline{r}(\tau_1)$ so that the lemma may be applied). Then we may show by induction that for all $k\in\mathbb{N}$,
\begin{equation}\label{Induction}
\Tr\Big[B(\U;\B_{x,\tau_0^{k}\tau_1r})\Big]\leq \delta^2\eps(\tau_0)\tau_0^{k(n-\frac{1}{2})}(\tau_1r)^{n-1}.\end{equation}

Indeed \eqref{Induction} implies that $\Hs(J_{\U}\cap \B_{\tau_0^k\tau_1 r})\leq \eps(\tau_0)(\tau_0^k\tau_1r)^{n-1}$, so with the same dichotomy as above we may apply the lemma \ref{DecayLemma} again to obtain \eqref{Induction} by induction.\bigbreak

Overall this means that $\frac{1}{\rho^{n-1}}\left(\int_{\B_{x,\rho}}|\nabla\U|^2\mathrm{d}\Ln +\Hs(J_\U\cap\B_{x,r})\right)\underset{\rho\to 0}{\rightarrow}0$, which is not the case when $x\in J_\U$ (see \cite{DCL89}, Theorem 3.6), so \eqref{Ahlfors} holds. By definition it also holds for $x\in \overline{J_{\U}}$ with a smaller constant, however according to \cite{DCL89}, lemma 2.6, $\Hs$-almost every $x$ such that $\liminf_{r\to 0}\frac{\Hs(J_{\U}\cap\B_{x,r})}{r^{n-1}}>0$ is in $J_{\U}$, which ends the proof.
\end{proof}

As a consequence of the existence of a relaxed minimizer and the regularity of relaxed minimizers, we obtain the theorem \ref{main1}.

\begin{proof}
We know from the proposition \ref{existence} that there exists a relaxed minimizer $\U$ of $\F$ in $\Uk(m)$, and from lemma \ref{PenalizationLemma} that $\U$ is an internal relaxed minimizer of $\F_\gamma$ for some $\gamma>0$ that only depends on the parameters. From the proposition \ref{apriori} we obtain that for certain constants $\delta,M,R>0$ only depending on the parameters, $\delta 1_{\{\U\neq 0\}}\leq |\U|\leq M$ and the diameter of the support of $\U$ (up to translation of its components) is less than $R$. From proposition \ref{regularity} we know that $\Hs(\overline{J_\U}\setminus J_\U)=0$. Since $|\U|\geq \delta 1_{\{\U\neq 0\}}$, we obtain
\begin{align*}
\Hs(\overline{J_\U})&=\Hs(J_\U)\leq \delta^{-2}\int_{J_{\U}}\left(|\overline{\U}|^2+|\underline{\U}|^2\right)\mathrm{d}\Hs\\
&\leq \beta^{-1}\delta^{-2}\left(\lambda_1(\U;\beta)+\hdots+\lambda_k(\U;\beta)\right)\leq C(n,m,\beta,F).
\end{align*}
Let $\Om$ be the union of the connected components of $\R^n\setminus \overline{J_\U}$ on which $\U$ is not zero almost everywhere. By definition $\partial\Om=\overline{J_{\U}}$, and $\U$ is continuous on $\Rn\setminus J_\U$ and do not take the values $\pm\frac{\delta}{2}$, thus $|\U|\geq \delta$ on $\Om$. In particular, $\{\U\neq 0\}$ and $\Om$ differ by a $\Ln$-negligible set, and $J_\U\subset\partial\Om$, so $\U_{|\Om}\in H^1(\Om)^k$. This means that for every $i=1,\hdots,k$, $\lambda_i(\Om;\beta)\leq \lambda_i(\U;\beta)$, so $\Om$ is optimal for $\F$.\bigbreak

In the proof of proposition \ref{regularity} we obtained the existence of a certain $\eps_1,r_1>0$ such that for every $x\in\partial\Om(=\overline{J_\U})$, $r<r_1$, then $\Hs(\B_{x,r}\cap\partial\Om)\geq \eps_1r^{n-1}$. By comparing $\U$ with $\U1_{\Rn\setminus B_{x,r}}$ (similarly to what was done in the proof of the proposition \ref{apriori}), we obtain the upper bound $\Hs(\B_{x,r}\cap\partial\Om)\leq Cr^{n-1}$; this concludes the proof.
\end{proof}

\section{The functional $\Om\mapsto\lambda_k(\Om;\beta)$}

We are now interested by the specific functional
\[\Om\mapsto \lambda_k(\Om;\beta).\]
While it is not covered by the previous existence result, relaxed minimizers of this functional were shown to exist in \cite{BG19}. To understand its regularity, it might be tempting to consider a sequence of relaxed minimizers with the function $F(\lambda_1,\hdots,\lambda_k)=\lambda_k+\eps (\lambda_1+\hdots+\lambda_{k-1})$ where $\eps\to 0$, however while the $L^\infty$ bound does not depend on $\eps$, the lower bound does and it seems to degenerate to 0 as $\eps$ goes to 0.\\
This prevents us to obtain any regularity on relaxed minimizers of this functional. We are, however, able to treat the specific case where the $k$-th eigenvalue would be simple, and this analysis allows us to prove that this does not happen in general. In particular, we shall prove $\lambda_k(\U;\beta)=\lambda_{k-1}(\U;\beta)$.\\

\subsection{Regularization and perturbation lemma}

We begin with a density result that allows us to suppose without loss of generality that $\U$ is bounded in $L^\infty$. This relies on the same procedure as \cite[Theorem 4.3]{BG19}.\\
We remind the notation for admissible functions used previously:
\[\Uk(m)=\left\{ \V\in\Uk:\ \V\text{ is linearly independant and }|\{\V\neq 0\}|=m\right\},\]
as well as the fact that if $\U$ is a relaxed minimizer of $\lambda_k(\cdot;\beta)$ in $\Uk(m)$ then according to lemma \ref{PenalizationLemma} there is a constant $\gamma>0$ such that $\U$ is a minimizer of
\[\V\mapsto \lambda_k(\V;\beta)+\gamma|\{\V\neq 0\}|\]
for linearly independant function $\V$ such that $|\{\V\neq 0\}|\in ]0,m]$.
\begin{lemma}\label{LemmaApriori}
Let $\U=(u_1,\hdots,u_k)$ be a relaxed minimizer of $\lambda_k(\U;\beta)$ in $\Uk(m)$. Suppose that $\lambda_k(\U;\beta)>\lambda_{k-1}(\U;\beta)$. Then there exists another minimizer $\V\in \Uk(m)$ that is linearly independant, normalized, such that $v_1\geq 0$, $\V\in L^\infty(\Rn)$, and
\[\lambda_{k-1}(\V;\beta)<\lambda_{k}(\V;\beta).\]
\end{lemma}
This justifies that in all the following propositions we may suppose that $\U\in L^\infty(\Rn)$ without loss of generality.
\begin{proof}
Without loss of generality suppose that $\U$ is normalized. Then according to \cite{BG19}, which itself relies on the Cortesani-Toader regularization (see \cite{CT99}), there exists a sequence of bounded polyhedral domains $(\Om^p)$ along with a sequence $\U^p\in \Uk\cap H^1(\Om^p)^k$ such that $\U^p\underset{p\to\infty}{\rightarrow}\U$ in $L^2$, and
\[\limsup_{p\to\infty}B(\U^p)\leq B(\U),\ \limsup_{p\to\infty}|\Om^p|\leq |\left\{\U\neq 0\right\}|.\]
Let $\V^p=(v_1^p,\hdots,v_k^p)$ be the first $k$ eigenfunctions of $\Om^p$ (with an arbitrary choice in case of multiplicity; notice $v_1^p$ may be chosen positive), then $B(\V^p)\leq B(\U^p)$ and with Moser iteration $\V^p$ is bounded in $L^\infty$ by $C_{n,\beta,m}\lambda_k(\U;\beta)^\frac{n}{2}$ (which, in particular, does not depend on $p$). Using the compactness result \ref{CompactnessLemma}, we find that up to an extraction $\V^p$ converges in $L^2$ and almost everywhere to $\V\in \Uk$ with lower semi-continuity on its Dirichlet-Robin energy, thus $\V$ is a minimizer in $L^\infty$ with $v_1\geq 0$. Moreover,
\[\lambda_{k-1}(\V;\beta)\leq \liminf_{p\to\infty}\lambda_{k-1}(\V^p;\beta)\leq \liminf_{p\to\infty}\lambda_{k-1}(\U^p;\beta)\leq \lambda_{k-1}(\U;\beta)<\lambda_{k}(\U;\beta)\leq \lambda_k(\V;\beta).\]
\end{proof}

\begin{lemma}\label{LemmaPerturb}
Let $\U=(u_1,\hdots,u_k)\in \Uk(m)\cap L^\infty(\Rn)$ be an internal relaxed minimizer of $\lambda_k(\U;\beta)$ in $\Uk(m)$, that we suppose to be normalized. Suppose that $\lambda_k(\U;\beta)=\lambda_{k-l+1}(\U;\beta)>\lambda_{k-l}(\U;\beta)$. Then there exists $\delta,\gamma>0$ such that, for all $\om\subset \mathbb{R}^n$ that verify
\[|\om|+\Per(\omega;\Rn\setminus J_\U)<\delta,\]
there exists $\alpha\in (\left\{ 0\right\}^{k-l}\times \mathbb{R}^l)\cap\mathbb{S}^{k-1}$ such that
\begin{equation}\label{EstPerturb}
\int_{\om}|\nabla u_\alpha|^2\mathrm{d}\Ln +\beta\int_{J_\U}\left(\overline{u_\alpha1_{\om}}^2+\underline{u_\alpha1_{\om}}^2\right)\mathrm{d}\Hs+\gamma|\om|\leq 2\beta \int_{\partial^*\om\setminus J_\U}u_\alpha^2\mathrm{d}\Hs+2\lambda_k(\U;\beta)\int_{\om}u_\alpha^2\mathrm{d}\Ln.\end{equation}
\end{lemma}

As may be seen in the proof, the factors $2$ on the right-hand side may be replaced by $1+\underset{\delta\to 0}{o}(1)$, however this will not be useful for us.\\
This result will only be applied in the particular case where $l=1$: when $l>1$ it gives a very weak information on the eigenspace of $\lambda_k(\U;\beta)$ and it would be interesting to see if the regularity of one of the eigenfunctions might be deduced from it as was done in \cite{BMPV15} (in the same problem with Dirichlet boundary conditions). In this case better estimates were obtained by perturbing the functional into $(1-\eps)\lambda_{k}+\eps\lambda_{k-1}$, considering a minimizer $\Om^\eps$ that contains the minimizer $\Om$ of $\lambda_k$, and separating the cases where $\lambda_k(\Om^\eps)$ is simple or not. However these arguments use crucially the monotonicity and scaling properties of $\lambda_i$, which are not available for Robin boundary conditions.

\begin{proof}
Let us denote $\V=\U 1_{\mathbb{R}^n\setminus\omega}$, $A,B=A(\V),B(\V)$, and for any $\alpha,\beta\in\R^k$,
\begin{align*}
A_{\alpha,\beta}&=\sum_{i=1}^k \alpha_i \beta_i A_{i,j},\\
B_{\alpha,\beta}&=\sum_{i=1}^k \alpha_i \beta_i B_{i,j}.\\
\end{align*}
We study the quantity
\[\lambda_k(\V;\beta)=\max_{\alpha\in\mathbb{S}^{k-1}}\frac{B_{\alpha,\alpha}}{A_{\alpha,\alpha}}.\]
Due to the $L^\infty$ bound on $\U$ and the fact that $|\om|+\Per(\om;\Rn\setminus J_\U)\leq \delta$:
\begin{align*}
\inf_{\alpha\in \left\{0\right\}^{k-l}\times \R^{l}\cap\mathbb{S}^{k-1}}\frac{B_{\alpha,\alpha}}{A_{\alpha,\alpha}}&\underset{\delta\to 0}{\longrightarrow}\lambda_{k}(\U;\beta),\\
\sup_{\eta\in\R^{k-l}\times \left\{0\right\}^{l}\cap\mathbb{S}^{k-2}}\frac{B_{\eta,\eta}}{A_{\eta,\eta}}&\underset{\delta\to 0}{\longrightarrow}\lambda_{k-l}(\U;\beta)(<\lambda_{k}(\U;\beta))
\end{align*}

Thus for a small enough $\delta$ the maximum above is attained for a certain $\frac{\alpha+t\eta}{\sqrt{1+t^2}}$ where $\alpha\in \left\{0\right\}^{k-l}\times \R^{l}\cap\mathbb{S}^{k-1}$, $\eta\in\R^{k-l}\times \left\{0\right\}^{l}\cap\mathbb{S}^{k-1}$ and $t\in\R$. $\alpha$ and $\eta$ are fixed in what follows and so
\[\lambda_k(\V;\beta)=\max_{t\in\mathbb{R}}\frac{B_{\alpha,\alpha}+2tB_{\alpha,\eta}+t^2B_{\eta,\eta}}{A_{\alpha,\alpha}+2tA_{\alpha,\eta}+t^2A_{\eta,\eta}}.\]
We let
\begin{align*}
b_{\alpha,\eta}&=\frac{B_{\alpha,\eta}}{B_{\alpha,\alpha}},\ &b_{\eta,\eta}=\frac{B_{\eta,\eta}}{B_{\alpha,\alpha}},\\
a_{\alpha,\eta}&=\frac{A_{\alpha,\eta}}{A_{\alpha,\alpha}},\ &a_{\eta,\eta}=\frac{A_{\eta,\eta}}{A_{\alpha,\alpha}},\\
\end{align*}
\[F(t)=\frac{1+2tb_{\alpha,\eta}+t^2b_{\eta,\eta}}{1+2ta_{\alpha,\eta}+t^2a_{\eta,\eta}}.\]
Then we may rewrite
\begin{equation}\label{eqlambdak}
\lambda_k(\V;\beta)=\frac{B_{\alpha,\alpha}}{A_{\alpha,\alpha}}\max_{t\in\mathbb{R}}F(t).\end{equation}
Moreover,
\[a_{\eta,\eta}\underset{\delta\to 0}{\longrightarrow}1,\ \limsup_{\delta\to 0}b_{\eta,\eta}\leq \frac{\lambda_{k-l}(\U;\beta)}{\lambda_{k}(\U;\beta)}<1.\]

We look for the critical points of $F$; $F'(t)$ has the same sign as
\[(a_{\alpha,\eta} b_{\eta,\eta}-a_{\eta,\eta} b_{\alpha,\eta})t^2-(a_{\eta,\eta}-b_{\eta,\eta})t+(b_{\alpha,\eta}-a_{\alpha,\eta}).\]
Since $F$ has the same limit in $\pm\infty$, this polynomial has two real roots given by:
\[t^{\pm}=\frac{a_{\eta,\eta}-b_{\eta,\eta}}{2(a_{\alpha,\eta}b_{\eta,\eta} - a_{\eta,\eta} b_{\alpha,\eta})}\left(1\pm\sqrt{1-4\frac{(b_{\alpha,\eta}-a_{\alpha,\eta})(a_{\alpha,\eta}b_{\eta,\eta} - a_{\eta,\eta} b_{\alpha,\eta})}{(a_{\eta,\eta}-b_{\eta,\eta})^2}}\right).\]
Since $F'$ has the same sign as $(a_{\alpha,\eta} b_{\eta,\eta}-a_{\eta,\eta} b_{\alpha,\eta})$ in $\pm\infty$, we find that the maximum of $F$ is attained in $t^{-}$. For any small enough $\delta$ we obtain
\[|t^{-}|\leq C_1|a_{\alpha,\eta}-b_{\alpha,\eta}|,\]
where $C_1$ only depends on $\lambda_k(\U;\beta),\lambda_{k-1}(\U;\beta)$. We evaluate $F$ in $t^-$ to obtain, for small enough $\delta$,
\[F(t^-)\leq 1+C_2(A_{\alpha,\eta}^2+B_{\alpha,\eta}^2),\]
where $C_2$ is another such constant. With the Cauchy Schwarz inequality we obtain
\begin{align*}
A_{\alpha,\eta}^2&=\underset{\delta\to 0}{o}\left(\int_{\om}u_\alpha^2\mathrm{d}\Ln\right),\\
B_{\alpha,\eta}^2&=\underset{\delta\to 0}{o}\left(\int_{\om}|\nabla u_\alpha|^2+\beta \int_{J_\U\cup \partial^*\om}\left(\underline{u_\alpha1_{\om}}^2+\overline{u_\alpha1_{\om}}^2\right)\mathrm{d}\Hs\right).
\end{align*}

Moreover,
\begin{align*}
B_{\alpha,\alpha}&=B(\U)_{\alpha,\alpha}-\int_{\om}|\nabla u_\alpha|^2\mathrm{d}\Ln-\beta\int_{J_\U}\left(\underline{u_\alpha1_{\om}}^2+\overline{u_\alpha1_{\om}}^2\right)\mathrm{d}\Hs+\int_{\partial^*\om\setminus J_\U}u_\alpha^2 \mathrm{d}\Hs,\\
A_{\alpha,\alpha}&=1-\int_{\om}u_\alpha^2\mathrm{d}\Ln.\end{align*}

Thus for a small enough $\delta$, we obtained the following estimate in \eqref{eqlambdak}
\begin{align*}
\left(1-\int_{\om}u_\alpha^2\mathrm{d}\Ln\right)\lambda_k(\V;\beta)&\leq B(\U)_{\alpha,\alpha}-(1-\underset{\delta\to 0}{o}(1))\left(\int_{\om}|\nabla u_\alpha|^2\mathrm{d}\Ln+\beta\int_{J_\U}\left(\underline{u_\alpha1_{\om}}^2+\overline{u_\alpha1_{\om}}^2\right)\mathrm{d}\Hs\right)\\
&+(1+\underset{\delta\to 0}{o}(1))\int_{\partial^*\om\setminus J_\U}u_\alpha^2 \mathrm{d}\Hs+\underset{\delta\to 0}{o}\left(\int_{\om}u_\alpha^2\mathrm{d}\Ln\right).\end{align*}
The optimality condition on $\U$ ($\lambda_k(\U;\beta)+\gamma |\om|\leq \lambda_k(\V)$ for a certain $\gamma>0$ that does not depend on $\om$, obtained through Lemma \ref{PenalizationLemma}) coupled with the fact that $\lambda_k(\U)=B(\U)_{\alpha,\alpha}$ gives us the estimate \eqref{EstPerturb} for any small enough $\delta$.
\end{proof}

\subsection{Non-degeneracy lemma and the main result}

\begin{proposition}
Let $\U=(u_1,\hdots,u_k)\in \Uk\cap L^\infty(\Rn)$ an internal relaxed minimizer of $\lambda_k(\U;\beta)+\gamma|\left\{ \U\neq 0\right\}|$. Suppose $n\geq 3$, and that $\lambda_k(\U;\beta)>\lambda_{k-1}(\U;\beta)$. Then there exists $c>0$ such that $|u_k|\geq c 1_{\left\{ u_k\neq 0\right\}}$.
\end{proposition}

\begin{proof}
We actually prove that there exists $r,t>0$ such that for any $x\in\Rn$, $|u_k|\geq t 1_{\B_{x,r}\cap\left\{ u_k\neq 0\right\}}$, since this is sufficient to conclude. We suppose $x=0$ to simplify the notations. We cannot proceed as in the proof of result \ref{apriori} because we do not know whether $\Per(\left\{|u_k|>t\right\};\Rn\setminus J_\U)$ is less than a constant $\delta$ or not. The idea is to compare $\U$ with $\U1_{\mathbb{R}^n\setminus \om_t}$ where
\[\om_t=\B_{r(t)}\cap\left\{ |u_k|\leq t\right\},\]
for $t>0$ and $r(t)>0$ chosen sufficiently small such that $\Per(\om_t;\Rn\setminus J_u)$ is sufficiently small.

\begin{lemma}\label{Lemma_Estimate}
Under these circumstances, there exists $t_1>0$ such that for all $t<t_1$,
\begin{equation}\label{Estimate}
\int_{\om_{t}}|\nabla u_k|^2\mathrm{d}\Ln +\beta\int_{J_\U}\left(\overline{u_k1_{\om_{t}}}^2+\underline{u_k1_{\om_{t}}}^2\right)\mathrm{d}\Hs+\frac{1}{2}\gamma|\om_{t}|\leq 2\beta \int_{\partial^*\om_{t}\setminus J_\U}u_k^2\mathrm{d}\Hs,\end{equation}
where $\om_t=\left\{|u_k|\leq t\right\}\cap \B_{r(t)}$ with $r(t):=\epsilon t^{\frac{2}{n-1}}$ for a small enough $\epsilon>0$.
\end{lemma}

\begin{proof}
As we said previously, this estimate will be obtained by comparing $\U$ and $\U1_{\Rn\setminus \om_t}$ where $\om_t=\B_{r(t)}\cap\left\{ |u_k|\leq t\right\}$. This is direct if we can apply Lemma \ref{LemmaPerturb}, we only need to show the hypothesis
\[\Hs(\partial^*\om_t\setminus J_u)< \delta.\]
Suppose that
\[\beta t^2\Hs(\partial^*\left\{|u_k|\leq t\right\}\cap\B_{r(t)}\setminus J_\U)\leq \int_{\om_{t}}|\nabla u_k|^2\mathrm{d}\Ln +\beta\int_{J_\U}\left(\overline{u_k1_{\om_{t}}}^2+\underline{u_k1_{\om_{t}}}^2\right)\mathrm{d}\Hs+\gamma|\om_{t}|.\]
Indeed if this inequality is false then we obtained the result. Then, comparing $\U$ with $\U1_{\Rn\setminus \B_{r(t)}}$ with the lemma \ref{LemmaPerturb} (which is allowed for any small enough $r>0$) we obtain the estimate
\begin{align*}
\int_{\B_{r(t)}}|\nabla u_k|^2\mathrm{d}\Ln +\beta\int_{J_\U}\left(\overline{u_k1_{\B_{r(t)}}}^2+\underline{u_k1_{\B_{r(t)}}}^2\right)\mathrm{d}\Hs+\frac{1}{2}\gamma|\B_{r(t)}|&\leq 2\beta \int_{\partial \B_{r(t)}\setminus J_\U}u_k^2\mathrm{d}\Hs\\
&\leq C(n,\beta,\Vert u_k\Vert_{L^\infty})r(t)^{n-1}.
\end{align*}
Combining the two previous inequalities,

\begin{align*}\Hs(\partial^*\left\{|u_k|\leq t\right\}\cap\B_{r(t)}\setminus J_\U)&\leq C(n,\beta,\Vert u_k\Vert_{L^\infty})\frac{r(t)^{n-1}}{t^2}\\
&=C(n,\beta,\Vert u_k\Vert_{L^\infty})\epsilon^{n-1}\\
&\leq \frac{1}{2}\delta\text{ for a small enough }\epsilon.\end{align*}
And so Lemma \ref{LemmaPerturb} may be applied, concluding the proof.
\end{proof}

We introduce the sets
\[\om^{\text{sup}}=\left\{ x:|u_k(x)|\geq |x/\eps|^{\frac{n-1}{2}}\right\},\ \om^{\text{inf}}=\left\{ x:|u_k(x)|\leq |x/\eps|^{\frac{n-1}{2}}\right\}\]
and the function
\[f(t)=\int_{\om_{t}}\left(|\nabla u_k|1_{\om^\text{sup}}+1_{\om^{\text{inf}}}\right)|u_k|\mathrm{d}\Ln.\]
From the coarea formula we get
\[f(t)=\int_0^t \left(\int_{\partial^*\left\{|u_k|\leq \tau\right\}\cap\B_{r(\tau)}\setminus J_\U}|u_k|\mathrm{d}\Hs\right)\mathrm{d}\tau+\int_{0}^{r(t)}\left(\int_{\partial\B_r\cap\left\{ |u_k|\leq (r/\eps)^{\frac{n-1}{2}}\right\}\setminus J_\U}|u_k|\mathrm{d}\Hs\right)\mathrm{d}r.\]
So $f$ is absolutely continuous and
\[f'(t)=\int_{\partial^*\{|u_k|\leq t\}\cap\B_{r(t)}\setminus J_\U}|u_k|\mathrm{d}\Hs+\frac{2\epsilon}{n-1}t^{-\frac{n-3}{n-1}}\int_{\partial\B_{r(t)}\cap \left\{ |u_k|\leq t\right\}\setminus J_\U}|u_k|\mathrm{d}\Hs.\]
We use here the fact that $n\geq 3$, so that for all small enough $t$ we get
\[\frac{1}{\epsilon}f'(t)\geq\int_{\partial^*\{|u_k|\leq t\}\cap\B_{r(t)}\setminus J_\U}|u_k|\mathrm{d}\Hs+\int_{\partial\B_{r(t)}\cap \left\{ |u_k|\leq t\right\}\setminus J_\U}|u_k|\mathrm{d}\Hs.\]
We will now estimate $f$ in a similar manner as in result \ref{apriori}.
\begin{align*}
c_n\left(\int_{\om_{t}}|u_k|^{2\frac{n}{n-1}}\mathrm{d}\Ln\right)^{\frac{n-1}{n}}&\leq D(|u_k|^21_{\om_{t}})(\mathbb{R}^n)\\
&=\int_{\om_{t}}2|u_k\nabla u_k|\mathrm{d}\Ln +\int_{J_\U}\left(\overline{u_k1_{\om_{t}}}^2+\underline{u_k1_{\om_{r,t}}}^2\right)\mathrm{d}\Hs\\
&+\int_{\partial^*\om_{t}\setminus J_u}|u_k|^2\mathrm{d}\Hs\\
&\leq |\om_{t}|+\int_{\om_{t}}|\nabla u_k|^2\mathrm{d}\Ln +\int_{J_\U}\left(\overline{u_k1_{\om_{t}}}^2+\underline{u_k1_{\om_{r,t}}}^2\right)\mathrm{d}\Hs\\
&+\int_{\partial^*\om_{t}\setminus J_u}|u_k|^2\mathrm{d}\Hs\\
&\leq C_{\beta,\gamma}\int_{\partial^*\om_{t}\setminus J_u}|u_k|^2\mathrm{d}\Hs\\
&\leq \frac{C_{\beta,\gamma}}{\epsilon}tf'(t).
\end{align*}
We used the lemma \ref{Lemma_Estimate} in the penultimate line, which is only valid for small enough $t$. The hypothesis that $n\geq 3$ was used in the last line. Finally,

\begin{align*}
f(t)&=\int_{\om_{t}}\left(|\nabla u_k|1_{\om^\text{sup}}+1_{\om^{\text{inf}}}\right)|u_k|\mathrm{d}\Ln\\
&\leq |\om_{t}|^{\frac{1}{2n}}\left(\int_{\om_{t}}|\nabla u_k|^2\mathrm{d}\Ln\right)^{\frac{1}{2}}\left(\int_{\om_{t}}|u|^{2\frac{n}{n-1}}\mathrm{d}\Ln\right)^{\frac{n-1}{2n}}+\gamma|\om_{t}|^{\frac{n+1}{2n}}\left(\int_{\om_{t}}|u_k|^{2\frac{n}{n-1}}\mathrm{d}\Ln\right)^{\frac{n-1}{2n}}\\
&\leq C_{n,\beta,\gamma}\left(tf'(t)\right)^{\frac{2n+1}{2n}},
\end{align*}
which implies for a certain $t>0$ that $f(t)=0$. Let $r=\eps t^{\frac{n-1}{2}}$, we show $|u_k|\geq t 1_{\B_{x,r}\cap\left\{ u_k\neq 0\right\}}$. From $f(t)=0$ we get that $u_k=0$ on $\B_r\cap \left\{ x:|u_k(x)|\leq |x/\eps|^{\frac{n-1}{2}}\right\}$. In particular, up to reducing slightly $r$ and $t$ we may suppose
\[\Hs(\partial\B_r\cap \left\{ |u_k|\leq t\right\})=0.\]
Moreover, $f(t)=0$ also gives that $\nabla u_k=0$ on $\B_r\cap\left\{ 0<u\leq t\right\}$. Consider $\U'=\U1_{\mathbb{R}^n\setminus \om}$ where $\om=\B_r\cap\left\{ |u_k|\leq t\right\}$. Note that $J_{\U'}\subset J_{\U}$, and for any small enough $t>0$,
\[\lambda_k(\U;\beta')\leq \lambda_k(\U;\beta)+2t^2|\om|-\frac{1}{2}\beta\int_{J_\U}\left(\overline{u_k1_{\om}}^2+\underline{u_k1_{\om}}^2\right)\mathrm{d}\Hs.\]
This contradicts the minimality of $\lambda_k(\U;\beta)+\gamma|\left\{ \U\neq 0\right\}|$ as soon as $|\om|>0$. This concludes the proof.\bigbreak

\end{proof}
Note that the proof fails when $n=2$ ; we need to choose $r(t)\ll t^{\frac{2}{n-1}}$ to ensure that the competitor $u1_{\Rn\setminus\om_t}$ yields information, but later we use $\inf_{t<1}r'(t)>0$ in a crucial way. When $n=2$ the inequalities are weakened to instead yield $f(t)\geq c t^5$, which is not enough to conclude.\bigbreak
We now deduce the second main result as a consequence.
\begin{proposition}
Suppose $n\geq 3,k\geq 2$. Let $\U=(u_1,\hdots,u_k)$ a relaxed minimizer of $\lambda_k(\U;\beta)$ in $\Uk(m)$. Then
\[\lambda_k(\U;\beta)=\lambda_{k-1}(\U;\beta).\]
\end{proposition}

\begin{proof}
Suppose that $\lambda_k(\U;\beta)>\lambda_{k-1}(\U;\beta)$. We may apply lemma \ref{LemmaApriori} to assume without loss of generality that $u_1\geq 0$ and $\U\in L^\infty$, so that all the previous estimates apply.\\
Let $\Om$ be the support of $u_k$, with $\Om^+=\left\{u_k>0\right\}$ and $\Om^- =\left\{u_k<0\right\}$.\bigbreak

We first notice that $|\left\{\U\neq 0\right\}\setminus\Om|=0$. Suppose indeed that it is not the case, and let $\om=\left\{\U\neq 0\right\}\setminus\Om$. Since $|u_k|\geq \delta 1_{\Om}$, $\U$ may be written as a disconnected sum of two $\Uk$ functions
\[\U=(\U1_\Om)\oplus(\U1_\om).\]
We may translate $\Om$ and $\om$ so that they have a positive distance from each other. Then consider $t>1$ and $s=s(t)<1$ chosen such that
\[|t\Om|+|s\om|=|\Om|+|\om|,\]
and $\U_t$ the function built by dilation of $\U$ on $t\Om\cup s\om$. Then for $t=1+\epsilon$ with a small enough $\epsilon$ we have $\lambda_k(\U_t;\beta)<\lambda_k(\U;\beta)$ with support of same measure, which is absurd by minimality of $\U$. Thus $|\om|=0$.\bigbreak

Since $u_1$ is nonnegative, has support in $\Om$, and $\langle u_1,u_k\rangle_{L^2}=0$, this means that $|\Om^+|,|\Om^-|>0$. We may again decompose $\U$ into
\[\U=(\U1_{\Om^+})\oplus(\U1_{\Om^-}).\]
Consider $\V\in \Up(m)$ for some $p\in \{k,\hdots,2k\}$ an extraction of $(\U1_{\Om^+},\U1_{\Om^-})$, such that it spans the same space in $L^2(\R^n)$ and $\V$ is linearly independant. Then for each $i\in\{1,\hdots,k\}$,
\[\lambda_i(\V;\beta)\leq \lambda_i(\U;\beta),\]
with equality if $i=k$ by optimality of $\U$. Since $A(\V)$ and $B(\V)$ are block diagonals we may suppose $\V$ is normalized such that its components have support in either $\Om^+$ or $\Om^-$: say $v_k$ is supported in $\Om^+$. This means that $\V=(v_1,\hdots,v_k)$ is a minimizer in $L^\infty$ such that $\lambda_k(\V;\beta)>\lambda_{k-1}(\V;\beta)$, and by the previous arguments we know that up to a negligible set $\{\V\neq 0\}\subset\{v_k\neq 0\}$, thus $|\Om^-|=0$: this is a contradiction.

\end{proof}

\subsection{Discussion about the properties of open minimizers}

Here we make a few observations on the properties of minimizing open sets, provided we know such sets exist.

\begin{proposition}
Let $\Om$ be an open minimizer of $\lambda_k(\Om;\beta)$ among opens sets of measure $m$, for $k\geq 2$, with eigenfunction $u_1,\hdots,u_k$. Suppose $\lambda_{k-l}(\Om;\beta)<\lambda_{k-l+1}(\Om;\beta)=\lambda_k(\Om;\beta)$. Then we know that
\[\cap_{i=k-l+1}^k \left(u_i^{-1}(\left\{0\right\})\cap\Om\right)=\emptyset.\]
In particular, for $k=3$ and $n=2$, $\Om$ is not simply connected.
\end{proposition}
\begin{proof}
By contradiction, consider $x\in \cap_{i=k-l+1}^k u_i^{-1}(\left\{0\right\})$, and $\U_r=(u_1,\hdots,u_k)1_{\R^n\setminus \B_{x,r}}$. For a small enough $r$, $\U_r$ is admissible and, with the same estimate as in Lemma \ref{LemmaPerturb}, there is a ($L^2$-normalized) eigenfunction $u_\alpha$ associated to $\lambda_k(\Om;\beta)$ such that
\[\int_{\B_{x,r}}|\nabla u_\alpha|^2\mathrm{d}\Ln+\gamma|\B_{x,r}|\leq 2\beta \int_{\partial\B_{x,r}}u_\alpha^2\mathrm{d}\Hs.\]
This implies that for any small enough $r>0$,
\[ \fint_{\partial\B_{x,r}}u_\alpha^2\mathrm{d}\Hs\geq \frac{r\gamma}{2n\beta}.\]
However, if $x$ is at the intersection of every nodal line associated to eigenfunctions of $\lambda_k(\Om;\beta)$, and since these eigenfunctions are $\mathcal{C}^1$, there is a constant $C>0$ such that for all $\alpha$, $|u_\alpha(y)|\leq C|x-y|$, thus
\[ \fint_{\partial\B_{x,r}}u_\alpha^2\mathrm{d}\Hs\leq C^2r^2,\]
which is a contradiction.\bigbreak

Let us now suppose that $n=2$, $k=3$, and that $\Om$ is simply connected. Since any eigenfunction related to $\lambda_3(\Om;\beta)$ has a non-empty nodal set, we know that
\[\lambda_1(\Om;\beta)<\lambda_2(\Om;\beta)=\lambda_3(\Om;\beta).\]
Let $u_1,u,v$ be the associated eigenfunctions. Every non-trivial linear combination of $u$ and $v$ is an eigenfunction associated to $\lambda_2(\Om;\beta)$ so it has a non-empty nodal set and no more than two nodal domains, thus, with the simple connectedness of $\Om$, its nodal set is connected (either a circle or a curve) and the eigenfunction changes sign at the nodal set.\\
Let us parametrize  the eigenspace with
\[w_t(x)=\cos(t)u(x)+\sin(t)v(x).\]
We show that the nodal sets $(\left\{w_t=0\right\})_{t\in \frac{\R}{\pi\mathbb{Z}}}$ are a partition of $\Om$ and that there is a continuous open function $T:\Om\to\frac{\R}{\pi\mathbb{Z}}$ such that $x\in \left\{w_{T(x)}=0\right\}$ for all $x\in\Om$. Indeed, the sets $(\left\{w_t=0\right\})_{t\in \frac{\R}{\pi\mathbb{Z}}}$ are disjoints because $u$ and $v$ have no common zeroes, and for any $x$ we may define

\[T(x)=-\text{arctan}\left(\frac{u(x)}{v(x)}\right),\]

where $\text{arctan}(\infty)=\frac{\pi}{2}[\pi]$ by convention. The function $T$ is continuous, $x\in \left\{w_{T(x)}=0\right\}$, and since eigenfunctions change sign at their nodal lines then $T$ is open. Since $\Om$ is simply connected $T$ may be lifted into
\[\Om\underset{T'}{\longrightarrow}\mathbb{R}\underset{p}{\longrightarrow}\mathbb{R}/\pi\mathbb{Z}.\]
Let $I$ be the image of $T'$, since $T$ is open, then $T'$ is too so $I$ is an open interval. If $T(x)=T(y)$, then $x$ and $y$ are in the same nodal line and since these are connected we know $T'(x)=T'(y)$. In particular, if $t$ is in $I$, then $t\pm \pi\notin I$; this implies that $I=]a,b[$ where $a<b$ and $b-a\leq\pi$. However every $w_t$ has a non-empty nodal set so $\frac{\R}{\pi\mathbb{Z}}=T(\Om)=p(]a,b[)$: this is a contradiction, thus $\Om$ is not simply connected.

\end{proof}

\bibliographystyle{plain}
\bibliography{biblio}

\end{document}